\documentclass[3p]{elsarticle}

\usepackage{lineno,hyperref}

\modulolinenumbers[5]

\journal{arXiv}

\usepackage{amsfonts}
\usepackage{amsmath, amscd,amssymb,bm,amsbsy,epsf}
    \usepackage{enumerate}
    \usepackage{paralist}
\usepackage{bbm, dsfont}
\usepackage{graphicx,color}
\usepackage{subfigure}
\usepackage{mathrsfs}
\usepackage{verbatim}
\usepackage{inputenc}
\usepackage{bbm}
\usepackage{multicol}
\usepackage{balance}
\usepackage{verbatim}
\usepackage{accents, booktabs}

\newtheorem{theorem}{Theorem}
\newtheorem{lemma}[theorem]{Lemma}
\newtheorem{corollary}[theorem]{Corollary}
\newtheorem{proposition}[theorem]{Proposition}
\newdefinition{definition}{Definition}
\newdefinition{hypothesis}{Hypothesis}
\newdefinition{remark}{Remark}
\newdefinition{example}{Example}
\newproof{proof}{Proof}

\DeclareMathOperator*{\interior}{int}

\DeclareMathOperator*{\mrad}{m_{\theta} }
\DeclareMathOperator*{\supp}{supp}
\DeclareMathOperator*{\Exp}{\bm{\mathbbm{E} } }
\DeclareMathOperator*{\prob}{\bm{\mathbbm{P} } }

\newcommand\thickbar[1]{\accentset{\rule{.4em}{.5pt}}{#1}}

\def\div{\bm{\nabla} \cdot}
\def\grad{\bm{\nabla}}
\def\del{\bm{\partial} }
\def\deled{\bm{\partial }_{e} }
\def\delsig{\bm{\partial }_{\sigma} }
\def\deli{\bm{\partial }_{i} }
\def\deledv{\bm{\partial }_{e, v} }
\def\wconv{\overset{w} \rightharpoonup}

\def\defining{\overset{\mathbf{def} } = }

\def\ind{\boldsymbol{\mathbbm{1}}}

\def\N{\bm{\mathbbm{N} } }
\def\R{\bm{\mathbbm{R}}}

\def\E{{\mathcal E} }
\def\G{{\mathcal G} }
\def\V{\mathcal{ V } }
\def\F{\thickbar{F} }
\def\Fi{\thickbar{F}_{i} }
\def\K{\thickbar{K} }
\def\h{\thickbar{h} }

\def\pn{p^{n}}
\def\pne{p^{n}_e}
\def\phone{\thickbar{p}^{\,n}_{1}}
\def\phtwo{\thickbar{p}^{\,n}_{2}}
\def\p{\thickbar{p} }
\def\plimi{\thickbar{p}_{i} }
\def\pe{p_e}
\def\psupe{p^{(e)}}
\def\hn{h^{n}}
\def\xin{\xi^{n}}

\begin{document}

\begin{frontmatter}

\title{Diffusion Processes Homogenization for Scale-Free Metric Networks }
\tnotetext[mytitlenote]{This material is based upon work supported by project HERMES 27798 from Universidad Nacional de Colombia,
Sede Medell\'in.}

\author{Fernando A Morales \sep $\&$ \sep Daniel E Restrepo 
}
\address{Escuela de Matem\'aticas
Universidad Nacional de Colombia, Sede Medell\'in \\
Calle 59 A No 63-20 - Bloque 43, of 106,
Medell\'in - Colombia}

\author[mymainaddress]{Fernando A Morales}

\cortext[mycorrespondingauthor]{Corresponding Author}
\ead{famoralesj@unal.edu.co}

%
%
\begin{abstract}
This work discusses the homogenization analysis for diffusion processes on scale-free metric graphs, using weak variational formulations. The oscillations of the diffusion coefficient along the edges of a metric graph induce internal singularities in the global system which, together with the high complexity of large networks constitute significant difficulties in the direct analysis of the problem. At the same time, these facts also suggest homogenization as a viable approach for modeling the global behavior of the problem. To that end, we study the asymptotic behavior of a sequence of boundary problems defined on a nested collection of metric graphs. This paper presents the weak variational formulation of the problems, the convergence analysis of the solutions and some numerical experiments.  
\end{abstract}

\end{frontmatter}


\section{Introduction}
%
%
A scale-free network is a large graph with power law degree distribution, i.e.,
$\prob[ \deg (v)= k ] \sim k^{-\gamma}$ where $\gamma$ is a fixed positive constant. Equivalently, the probability of finding a vertex of degree $k$, decays as a power-law of the degree value. Power-law distributed networks are of noticeable interest because they have been frequently observed in very different fields such as the World Wide Web, business networks, neuroscience, genetics, economics, etc. The current research on scale-free networks is mainly focused in three aspects: first, generation models (see \cite{BarabasiAlbert, Kumar}), second, solid evidence detection of networks with power-law degree distribution (see \cite{Wuchty, Faloutsos, Bortoluzzi, Tsaparas}). The third and final aspect studies the extent to which the power-law distribution relates with other structural properties of the network, such as self-organization (see \cite{Keller, LunLi}); this is subject of intense debate, see \cite{Keller} for a comprehensive survey on the matter.     
\vspace{10pt}

The present work studies scale-free networks from a very different perspective. Its main goal is to introduce a homogenization process on the network, aimed reduce the original order of complexity but preserving the essential features (see \textit{Figure} \ref{Fig Networks}). In this way, the ``homogenized" or ``upscaled" network is reliable for data analysis while, at the same time, involves lower computational costs and lower numerical instability. Additionally, the homogenization process derives a neater and more structural picture of the starting network since unnecessary complexity is replaced by the average asymptotic behavior of large data. The phenomenon known as ``The Aggregation Problem" in economics is an example of how this type of reasoning is implicitly applied in modeling the global behavior of large networks (see \cite{Kreps}). Usually, homogenization techniques require some assumptions of periodicity of the singularities or the coefficients of the system (see \cite{CioranescuDonato, Hornung}), in turn this case demands averaging hypotheses in the Ces\`aro sense. The resulting network has the desired features because of two characteristic properties of scale-free networks. On one hand, they resemble star-like graphs (see \cite{EstradaComplexNetworks}), on the other hand, they have a ``communication kernel" carrying most of the network traffic (see \cite{Kim}).  
\vspace{10pt}

This paper, for the sake of clarity, restricts the analysis to the asymptotic behavior of diffusion processes on stared metric graphs (see \textit{Definition} \ref{Def Metric Graph} and \textit{Figure} \ref{Fig Graphs Stages five and n} below). However, while most of the models in the preexistent literature are concerned with the strong forms of differential equations (see \cite{BerkolailoKuchment} for a general survey and \cite{JRamirez} for the stochastic modeling of advection-diffusion on networks), here we use the variational formulation approach, which is a very useful tool for upscaling analysis. More specifically, we introduce the the pseudo-discrete analogous of the classical stationary diffusion problem 
\begin{equation}\label{Pblm Dirchlet Classic Problem}
\begin{split}
-\div (K\nabla p) =f \quad \text{in} \; \Omega \, ,\\
  p = 0    \quad \text{on} \; \partial \Omega \, ,
  \end{split}
\end{equation}
where $K$ is the diffusion coefficient (see \textit{Definition} \ref{Def Strong Problem on Graphs} and \textit{Equation} \eqref{Pblm Weak Variational on Graphs} below). Due to the variational formulation it will be possible to attain a-priori estimates for a sequence of solutions, an asymptotic variational form of the problem and the computation of effective coefficients. Finally, from the technique, it will be clear how to apply the method to scale-free metric networks in general. 
\vspace{10pt}

Throughout the exposition the terms ``homogenized", ``upscaled" and ``averaged" have the same meaning and we use them indistinctly. The paper is organized as follows, in \textit{Section} \ref{Sec Preliminaries} the necessary background is introduced for $L^{2}$, $H^{1}$-type spaces on metric graphs as well as the strong form and the weak variational form, together with its well-posedness analysis. Also a quick review of equidistributed sequences and Weyl's Theorem is included to be used mostly in the numerical examples. In \textit{Section} \ref{Sec The Sequence of Problems} we introduce a geometric setting and a sequence of problems for its asymptotic analysis, a-priori estimates are presented and a type of convergence for the solutions. In \textit{Section} \ref{Sec Radial Approach}, under mild hypotheses of Ces\`aro convergence for the forcing terms, the existence and characterization of a limiting or homogenized problem are shown. Finally, \textit{Section} \ref{Sec Numerical Experiments} is reserved for the numerical examples and \textit{Section} \ref{Conclusions and Final Discussion} holds the conclusions.
\begin{figure}[h]\label{Fig Networks} 
        \centering
        \begin{subfigure}[Scale-Free Network. ]
                {\resizebox{5.3cm}{7.3cm}
                {\includegraphics{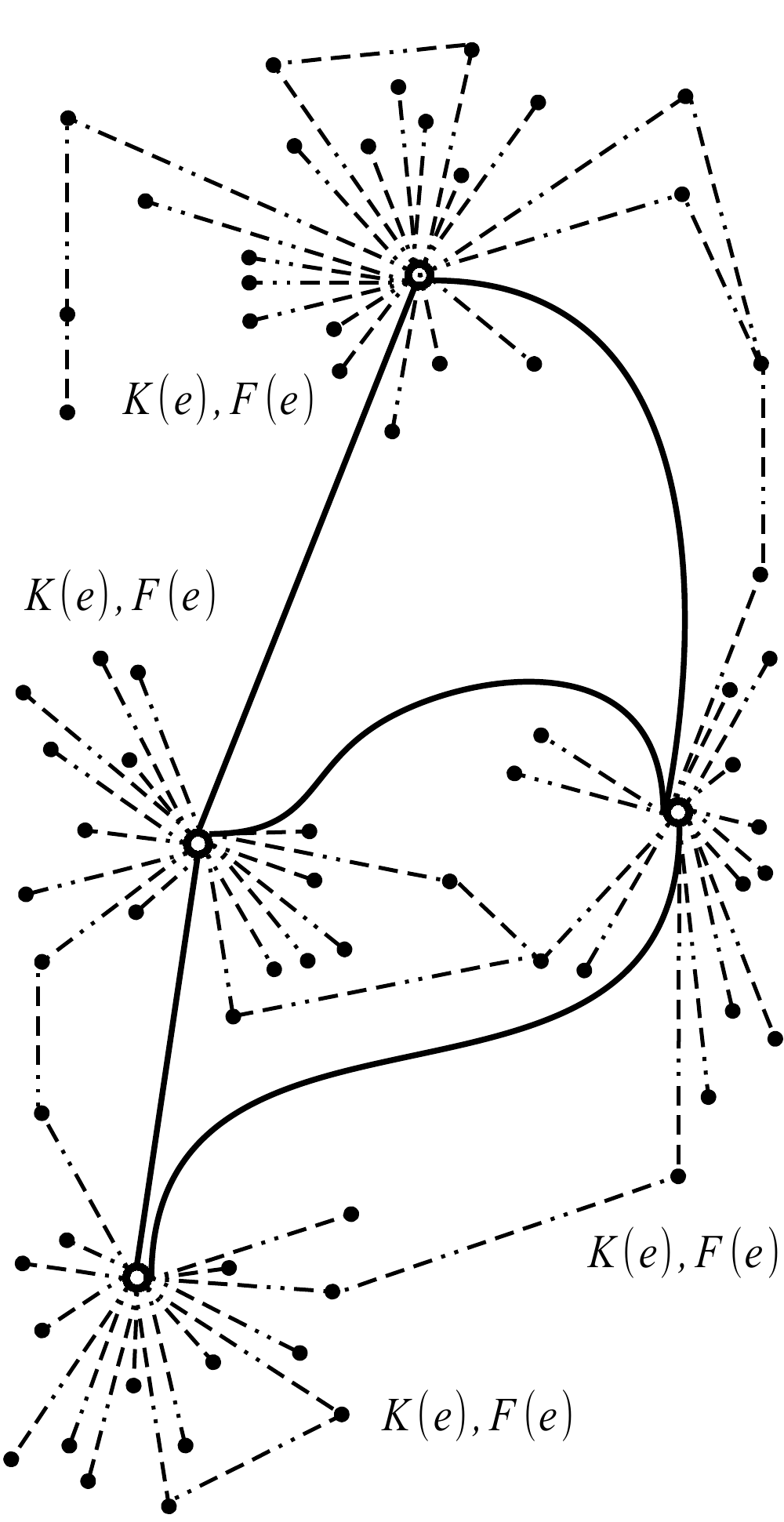} } }
        \end{subfigure} 
        \qquad
        ~ 
          \begin{subfigure}[Homogenized Network.]
                {\resizebox{5.3cm}{7.3cm}
                {\includegraphics{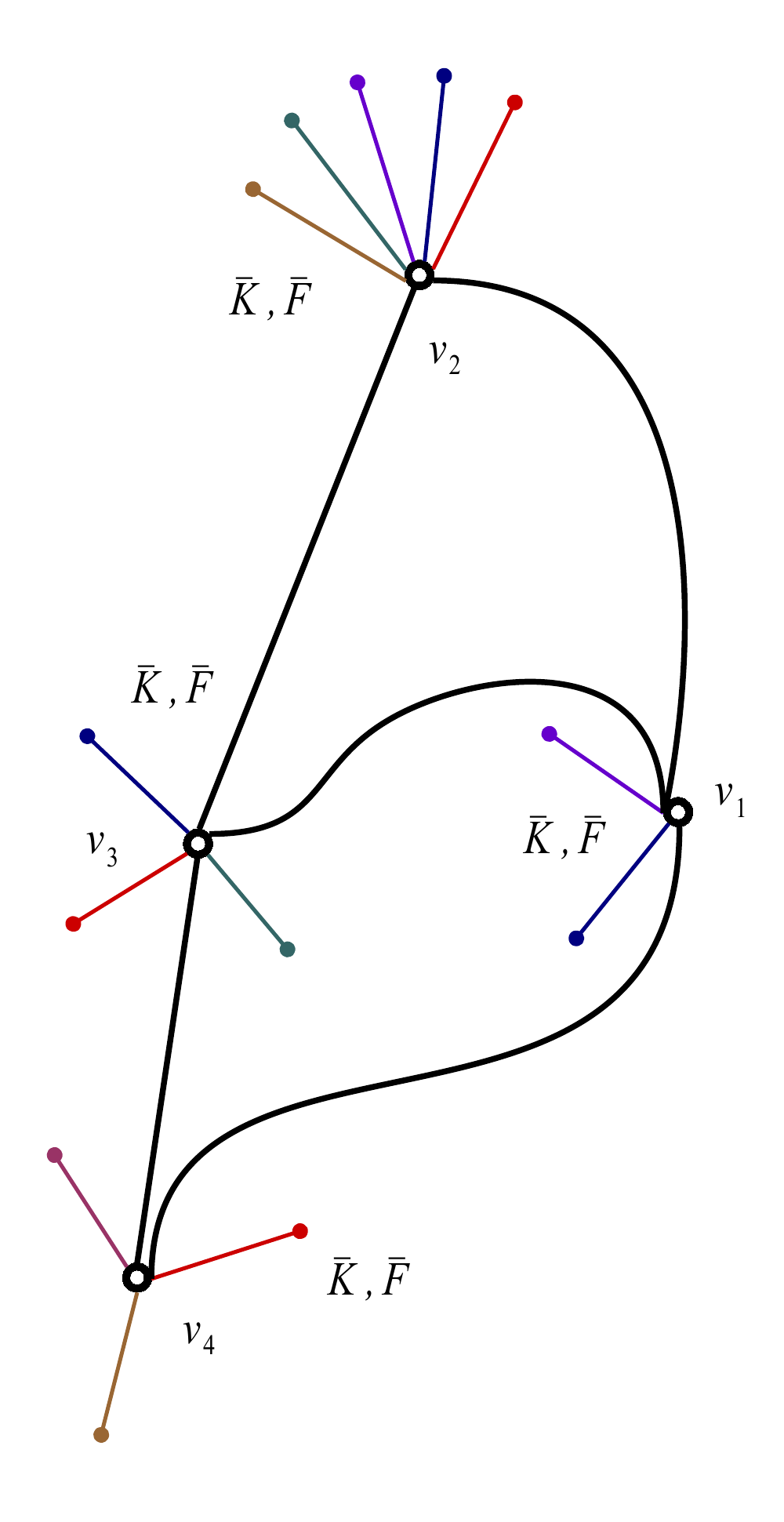} } }                
        \end{subfigure} 
\caption{Figure (a) depicts a scale-free network. Figure (b) depicts a homogenization of the original network.}
\end{figure}
%
%
%
%
\section{Preliminaries}\label{Sec Preliminaries}
%
%
%
%
%
%
\subsection{Metric Graphs and Function Spaces}
%
%
%
We begin this section recalling some facts for embeddings of graphs. 
\begin{definition}\label{Def planar graph}
A graph $G = (V, E)$ is said to be \textbf{embeddable} in $\R^{\! N}$ if it can be drawn in $\R^{\!N}$ so that its edges intersect only at their ends. A graph is said to be \textbf{planar} if it can be embedded in the plane. 
\end{definition}
It is a well-known fact that any simple graph can be embedded in $\R^{\!2}$ or $\R^{\!3}$ (depending whether it is planar or not) in a way that its edges are drawn with straight lines; see \cite{Fary1948} for planar graphs and \cite{BondyMurty} for non-planar graphs. In the following it will always be assumed that the graph is already embedded in $\R^{\!2}$ or $\R^{\!3}$.
%
%
%
%
\begin{definition}\label{Def Metric Graph}
Let $G = (V, E)$ be a graph embedded in $\R^{2}$ or $\R^{3}$, depending on the case.
\begin{enumerate}[(i)]
\item The graph $G$ is said to be a \textbf{metric graph} if each edge $e\in E$ is assigned a positive length $\ell_{e} \in (0, \infty)$.

\item The graph $G$ is said to be \textbf{locally finite} if $\deg(v) < + \infty$ for all $v\in V$.


\item If the graph $G$ is metric, the \textbf{boundary} of the graph is defined by the set of vertices of degree one. The set will also be denominated as the set of \textbf{boundary vertices} and denoted by
\begin{equation}\label{Def Boundary of a graph}
\partial V \defining \{ v\in V: \deg(v) = 1 \}.
\end{equation} 

\item Given a metric graph we define its natural domain by 
\begin{equation}\label{Def Domain of a graph}
\Omega_{G} \defining \bigcup_{e \, \in\, E } \interior (e) .
\end{equation}
\end{enumerate}
\end{definition}
%
%
%
%
\begin{definition}\label{Def Function Spaces on Metric Graphs}
Let $G = (V, E)$ be metric graph we define the following associated Hilbert spaces 
\begin{enumerate}[(i)]
\item The space of square integrable functions, or \emph{mass space} is defined by 
\begin{subequations}\label{Def square integrable functions}
\begin{equation}\label{Def square integrable functions direct sum}
L^{2}(G) \defining \bigoplus_{e \, \in \, E} L^{2}(e) , 
\end{equation}
endowed with its natural inner product
\begin{equation}\label{Def square integrable functions inner product}
\langle f, g\rangle_{L^{2}(G)} \defining \sum_{e \, \in \, E} \int_{e} f\, g .
\end{equation}
\end{subequations}

\item The \emph{energy space} of functions is defined by
\begin{subequations}\label{Def energy space functions}
\begin{multline}\label{Def energy space functions direct sum}
H^{1}(G) \defining \Big\{f\in \bigoplus_{e \, \in \, E} H^{1}(e) : 
 \lim_{\substack{x \, \rightarrow \, v\\
 x \, \in \, e} } f(x) = \lim_{\substack{x \, \rightarrow \, v\\
 x \, \in \, \sigma } }f(x) \, ,  \\
\text{for all vertices} \, v\in V \, \text{and for all edges }\; e, \sigma \;\text{ incident on}\; v 
\Big\} .
\end{multline}
In the sequel $f(v)\defining \lim\{ f(x): x \, \rightarrow \, v , \, x \, \in \, e \}$, with $e\in E$ any edge incident on $v$. We endow the space with its natural inner product
\begin{equation}\label{Def energy space functions inner product}
\langle f, g\rangle_{H^{1}(G)} \defining \sum_{e \, \in \, E} \int_{e} f\, g 
+ \sum_{e \, \in \, E} \int_{e} \deled f\, \deled g .
\end{equation}
\end{subequations}
Here $\deled$ denotes the derivative along the edge $e\in E$.

\item The space $H_{0}^{1}(G)$ is defined by

\begin{equation}\label{Eq null boundary conditions}
H_{0}^{1}(G) \defining \big\{f\in H^{1}(G) : 
f (v) =0\, , \; \text{for all} \; v\in \partial V\big\} ,
\end{equation}
endowed with the standard inner product \eqref{Def energy space functions inner product}.


\end{enumerate}
\end{definition}
\begin{remark}\label{Rem Identification of the Directional Derivative}
Let $G$ be a metric graph 
\begin{enumerate}[(i)]
\item Notice that the definition of $\deled$ is ambiguous in the expression \eqref{Def energy space functions inner product}, 
such ambiguity will cause no problems since the bilinear structure of the inner product is indifferent to the choice of direction
\begin{equation*}
(q,r)\mapsto 
 \deled q\, \deled  r=
\big(- \deled q \big) \, (-\deled r ). 
\end{equation*}
%


\item Whenever there is need to specify the direction of the derivate, we write $\deledv$ to indicate the direction pointing from the interior of the edge $e$ towards the vertex $v$ on one of its extremes. 


\item\label{Rem inner product equivalence} Notice that if the metric graph $G$ is connected, then the Poincar\'e inequality holds and the inner product 
\begin{equation}\label{Eq energy inner product}
(f, g)\mapsto \sum_{e\,\in \,E} \int_{e} \deled f\, \deled g ,
\end{equation}
is equivalent to the standard one \eqref{Def energy space functions inner product} in the space $H_{0}^{1}(G)$. 

\item Observe that the condition of agreement of a function $f\in H^{1}(G)$ on the vertices of the graph $G$ does not necessarily imply continuity as a function $f: \Omega_{G}\rightarrow \R$. For if the degree of a vertex $v\in V$ is infinite and the function is continuous on $v$, then it follows that the convergence $f(v)\defining \lim\{ f(x): x \, \rightarrow \, v , \, x \, \in \, e \}$ is uniform for all the edges $e$ incident on $v$. Such a condition can not be derived from the norm induced by the inner product \eqref{Def energy space functions inner product}, although the function $f\ind_{\interior(e)}$ is continuous for all $e\in E$.  
\end{enumerate}
\end{remark}
\begin{definition}\label{Def Increasing Graph Sequence}
Let $G_{n} = (V_{n}, E_{n})$ be a sequence of graphs. 
\begin{enumerate}[(i)]
\item The sequence $\{ G_{n}: n\in \N \}$ is said to be \textbf{increasing} if $V_{n}\subseteq V_{n+1}$ and $E_{n} \subseteq E_{n+1}$ for all $n\in \N$. 

\item Given an increasing sequence of graphs $\{ G_{n}: n\in \N \}$, we define the \textbf{limit graph} $G = (V, E)$ in the natural way i.e., 
\begin{subequations}
\begin{align*}
& V \defining \bigcup_{n\,\in\,\N} V_{n} &
& E \defining \bigcup_{n\,\in\,\N} E_{n} .
\end{align*} 
In analogy with monotone sequences of sets we adopt the notation
\begin{equation*}
G \defining \bigcup_{n\,\in\,\N} G_{n} .
\end{equation*} 
\end{subequations}
\end{enumerate}
\end{definition}
%
%
%
\subsection{The Strong and Weak Forms of the Stationary Diffusion Problem on Graphs}
%
%
\begin{definition}\label{Def Strong Problem on Graphs}
Let $G = (V, E)$ be a locally finite metric graph, $F\in L^{2}(G)$ and $h: V - \del V \rightarrow \R$, define the following diffusion problem
\begin{subequations}\label{Pblm Dirichlet on Graphs}
\begin{equation}\label{Def Strong Equation on Graphs}
\sum_{e \,\in \, E} -\deled \big(K\deled p \big) \ind_{e}
=\sum_{e \,\in \, E}  F\,\ind_{e} \quad \text{in} \; \Omega_{G}  .
\end{equation}
Where $K\in L^{\infty}(\Omega_{G})$ is a nonnegative diffusion coefficient. We endow the problem with normal stress continuity conditions
\begin{equation}\label{Def Strong Normal Stress on Graphs}
\lim_{\substack{x \, \rightarrow \, v\\
x \, \in \, e} } p(x) = p(v) \quad \text{for all } \; p\in V - \partial V ,
\end{equation}
and normal flux balance conditions
\begin{equation}\label{Def Strong Normal Flux on Graphs}
h (v) +
\sum_{\substack{e\,\in \,E\\ e\, \text{incident on} \, v } } 
\lim_{\substack{x \rightarrow v\\
x \, \in \, e} } \deledv \, p(x) =  0 \quad \text{for all } \; v\in V - \partial V .
\end{equation}
Here $\deledv$ denotes the derivative along the edge $e$ pointing away from the vertex $v$. Finally, we declare homogeneous Dirichlet boundary conditions
\begin{equation}\label{Def Strong Boundary Conditions on Graphs}
p(v)=0    \quad \text{for all } \; v \in  \partial V .
\end{equation}
\end{subequations}
\end{definition}
A weak variational formulation of this problem is given by
\begin{align}\label{Pblm Weak Variational on Graphs}
& p\in H^{1}_{0}(G):&
& \sum_{e \, \in \, E} \int_{e} K \deled p\, \deled q = 
\sum_{e \, \in \, E} \int_{e} F\,  q 
+ \sum_{v \, \in \, V - \partial V} h(v)\,  q(v) &
& \forall \, q\in H^{1}_{0}(G) .
\end{align}
For the sake of completeness we present the following standard result. 
\begin{proposition}\label{Th Well-Posedness Weak variational on Graphs}
Let $G  = (V, E)$ be a locally finite connected metric graph such that $\partial V \neq \emptyset$ and let $K\in L^{\infty}(\Omega_{G})$ be a diffusion coefficient such that $K(x) \geq c_{K} > 0$ almost everywhere in $\Omega_{G}$. Then the problem \eqref{Pblm Weak Variational on Graphs} is well-posed. 
\end{proposition}
\begin{proof}
Clearly the functionals on the right hand side of \textit{Problem} \eqref{Pblm Weak Variational on Graphs} are linear and continuous, as well as the bilinear form $b(p, q) \defining \sum_{e \, \in \, E}  \int_{e} K \deled p\, \deled q$ of the left hand side. Additionally, 
\begin{equation*}
\sum_{e \, \in \, E}  \int_{e} K \vert \deled p\vert^{2} \geq 
c_{K} \sum_{e \, \in \, E}  \int_{e}  \vert \deled p\vert^{2} \geq 
\tilde{c}  \sum_{e \, \in \, E}  \Vert  p\Vert_{H^{1}(e)} ^{2}
= \tilde{c} \,  \Vert p\Vert_{H^{1}(G)} ^{2} . 
\end{equation*}
The first inequality above holds due to the conditions on $K$. The second inequality hods due to the Dirichlet homogeneous boundary conditions and the connectedness of the graph $G$, which permits the Poincar\'e inequality on the space $H^{1}_{0}(G)$ as discussed in Remark \ref{Rem Identification of the Directional Derivative} \eqref{Rem inner product equivalence} above. Therefore, due to the Lax-Milgram Theorem, the Problem \eqref{Pblm Weak Variational on Graphs} is well-posed.
\qed
\end{proof}
%
%
%
\subsection{Equidistributed Sequences and Weyl's Theorem}
%
%
The brief review of equidistributed sequences and Weyl's theorem of this section will be applied, almost exclusively in the numerical examples below, see \textit{Section} \ref{Sec Numerical Experiments}. For a complete exposition on equidistributed sequences and Weyl's Theorem see \cite{SteinShakarchi}.
%
%
%
%
\begin{definition}\label{Def equidistributed sequences}
A sequence $\{\theta_n: n\in\N\}$ is called equidistributed on an interval $[a,b]$ if for each subinterval $[c,d]\subseteq [a,b]$ it holds that:
\begin{equation}
\lim\limits_{n\to \infty} \frac{ \#\{ i:\theta_i\in [c,d],1\leq i\leq n\}}{n}=\frac{d-c}{b-a} .
\end{equation}
\end{definition}
\begin{theorem}[Weyl's Theorem]\label{Th Weyl's Theorem}
	Let $\big\{\theta_n: n\in\N \big\}$ be a sequence on $[a,b]$, the following conditions are equivalent:
	\begin{enumerate}[(i)]
		\item The sequence $\{\theta_n: n\in\N\}$ is equidistributed in $[a,b]$.
		
		\item For every Riemann integrable function $f:[a,b]\to \bm{\mathbbm{C}}$ 
		\begin{equation*}
	\lim\limits_{n\to \infty}\frac{1}{n}\sum\limits_{i=1}^{n} f(\theta_i) 
	=\frac{1}{b-a}\int_{a}^{b}f(\theta)\, d \theta .
		\end{equation*}
	\end{enumerate}
\end{theorem}
%
%
%
\begin{definition}\label{Def angular average}
Let $\Omega = B(0, 1)\subseteq \R^{2}$ and let $f: \Omega\rightarrow \R$ be such that its restriction to every sphere $\partial B(0, \rho)$ with $0\leq \rho < 1$ is Riemann integrable. Then, we define its \textbf{angular average} by the average value of $f$ along the sphere $\partial (B(0, \rho))$, i.e.,
\begin{align}\label{Eq radial average}
& \mrad[f]: [0, 1)\rightarrow \R \, , &
& \mrad[f] (t)\defining \frac{1}{2 \pi}\int_{0}^{2\pi} f\big(t \cos \theta, \, t  \sin \theta) \big)\,d \theta .
\end{align}
%
%
\end{definition}
%
%
%
%
%
\section{The Sequence of Problems}\label{Sec The Sequence of Problems}
%
%
%
%
In this section we analyze the behavior of the solutions $\{\pn: n\in \N\}$ of a family of well-posed problems on an very particular increasing sequence of graphs $\{G_{n}: n\in \N \}$, depicted in \textit{Figure} \ref{Fig Graphs Stages five and n}. 
%
%
%
%
\subsection{Geometric Setting and the n-Stage Problem}
%
%
In the following we denote by $\Omega$, $S^{1}$ the unit disk and the unit sphere in $\R^{2}$ respectively. The function $F:\Omega_{G}\rightarrow \R$ is such that $F\ind_{\Omega_{n}} \in L^{2}(\Omega_{n})$ for all $n\in \N$, $\{\hn: n\in \N\} $ is a sequence of real numbers and the diffusion coefficient $K\in L^{\infty}(\Omega_{G})$ is such that $K(\cdot) \geq c_{K} > 0$ almost everywhere in $\Omega_{G}$.
\begin{definition}\label{Def radial equidistributed graph}
Let $\{v_{n}: n \geq 1 \}$ be an equidistributed sequence in $S^{1}$ and $v_{0}\defining 0 \in \R^{2}$. 
\begin{enumerate}[(i)]
\item For each $n\in \N$ define the graph $G_{n} = (V_{n}, E_{n})$ in the following way:
\begin{align}\label{Def n-stage vertices radial equidistributed graph}
& V_{n} \defining \{ v_{n}:0\leq i\leq n  \} , &
& E_{n} \defining \{v_{0}v_{i}: 1\leq i\leq n \}.
\end{align}
%

\item For the increasing sequence of graphs $\{G_{n}: n\in \N \}$ define the limit graph $G \defining \bigcup_{n\,\in\,\N} G_{n}$ as described in \textit{Definition} \ref{Def Increasing Graph Sequence}.

\item In the following we denote the natural domains corresponding to $G$, $G_{n}$ by $\Omega_{G}$ and $\Omega_{n}$ respectively. 

\item For any edge $e\in E$ we denote by $v_{e}$ its boundary vertex and $\theta_{e}\in [0, 2\pi]$ the direction of the edge. 

\item From now on, for each edge $e = v_{0} v_{e}$ and $f: e\rightarrow \R$ a function, it will be understood that its one-dimensional parametrization, is oriented from the central vertex $v_{0}$ to the boundary vertex $v_{e}$. Consequently the derivative $\deled$ equals $\del_{e, v_{e}}$. 

\item For any given function $f: \Omega_{G}\rightarrow \R$ (or $f: \Omega_{n}\rightarrow \R$) we denote by $f_{e}: (0,1)\rightarrow \R$, the real variable function $f_{e}(t) \defining (f\ind_{e})(t\cos \theta_{e}, t\sin \theta_{e})$ on the edges $e\in E$ (or $e\in E_{n}$ respectively).

\begin{figure}[h] 
        \centering
        \begin{subfigure}[Graph Stage $5$. ]
                {\resizebox{5.3cm}{5.3cm}
                {\includegraphics{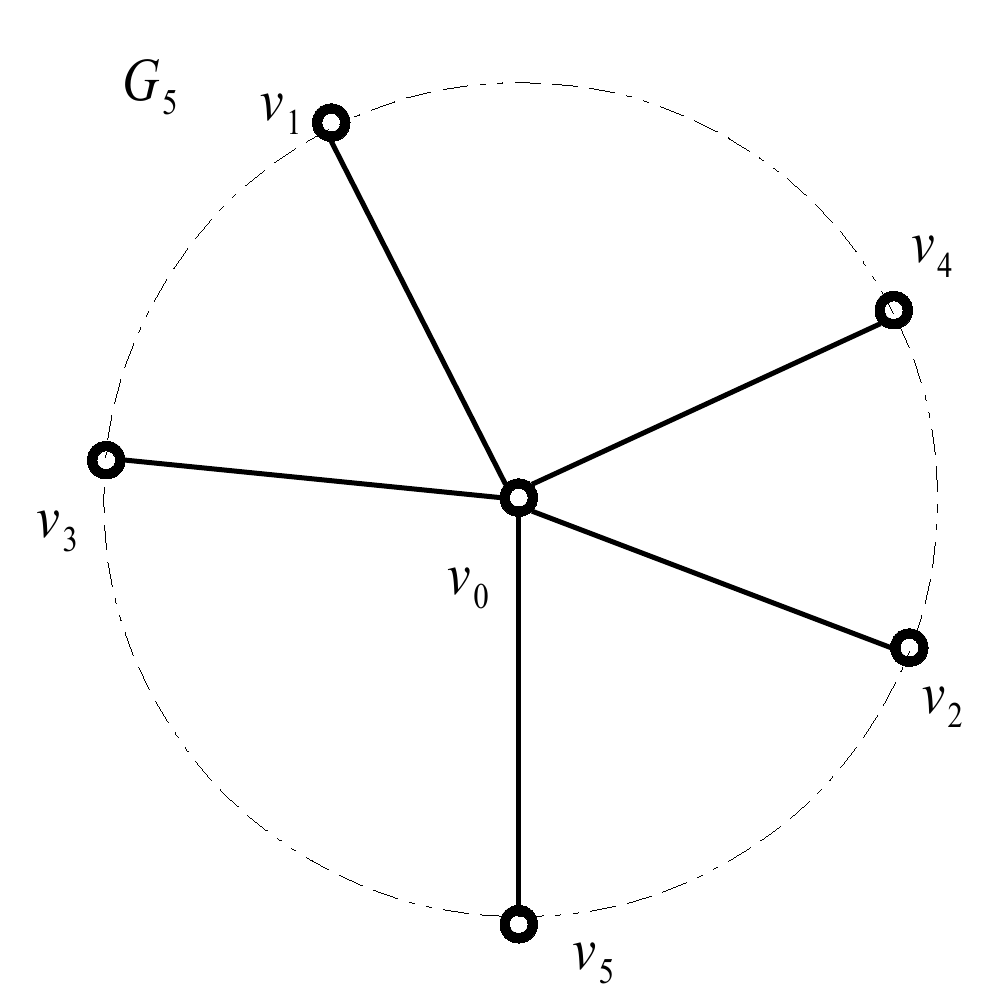} } }
        \end{subfigure} 
        \qquad
        ~ 
          \begin{subfigure}[Graph Stage n.]
                {\resizebox{5.3cm}{5.3cm}
                {\includegraphics{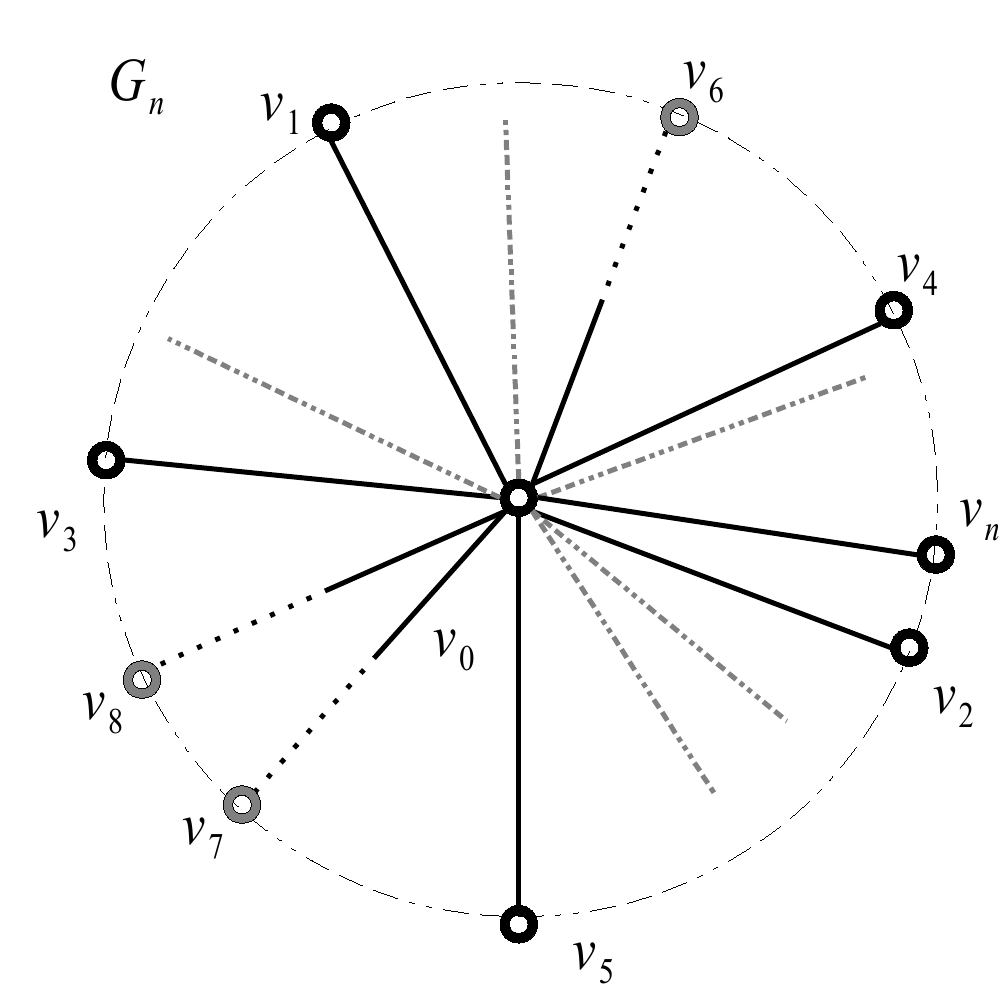} } }                
        \end{subfigure} 
%
\caption{Figure (a) depicts the stage $5$ of the graph $G$. Figure (b) depicts a more general stage $n$ of the graph $G$. \label{Fig Graphs Stages five and n} }
\end{figure}
\end{enumerate}
\end{definition}
\begin{remark}\label{Rem non-necessary equidistributed condition}
From the following analysis, it will be clear that it is not necessary to assume that the sequence of vertices $\{v_{n}: n\in \N \}$ of the graph is equidistributed or that the vertices are in $S^{1}$ or even that the graph is embedded in $\R^{2}$. We adopt these assumptions, mainly to facilitate a geometric visualization of the setting.
\end{remark}

From now on it will be assumed that $\{G_{n}: n\in \N \}$ is the increasing sequence of graphs, with $G$ its limit graph, as in the \textit{Definition} \ref{Def radial equidistributed graph} above.
Next, we define the family of problems to be studied, for each $n\in \N$ consider the well-posed problem
\begin{align}\label{Pblm n-Stage Weak Variational}
& \pn\in H^{1}_{0}(G_{n}):&
& \sum_{e \, \in \, E_{n}} \int_{e} K \deled \pn\, \deled q = 
\sum_{e \, \in \, E_{n}} \int_{e} F\,  q 
+ 
\hn \,  q(v_{0}) \,  ,&
& \forall \, q\in H^{1}_{0}(G_{n}) .
\end{align}
We are to analyze the asymptotic behavior of the sequence of solutions $\{\pn: n\in \N \}$. One of the main challenges is that the elements of the sequence are not defined on the same global space. The fact that $\pn(0)$ may not be zero makes impossible to extend this function to $H_{0}^{1}(G)$ directly, however it will play a central role in the asymptotic analysis of the problem.

\subsection{Estimates and Edgewise Convergence Statements}\label{Estimates and Edgewise Convergence Statements}
%
%
In this section we get estimates for the sequence of solutions, several steps have to be made as it is not direct to attain them. We start introducing conditions to be assumed from now on
\begin{hypothesis}\label{Hyp Forcing Terms and Permeability}
\begin{enumerate}[(i)]
\item The forcing term $F$ is defined in the whole domain, i.e. $F: \Omega_{G}\rightarrow \R$ and  $M \defining \sup_{e \, \in \,  E} \Vert F\Vert_{L^{2}(e)}< + \infty$.

\item The sequence $\big\{\dfrac{1}{n}\, \hn: n\in \N \big\}$ is bounded.

\item The permeability coefficient satisfies that  $K\in L^{\infty}(\Omega_{G})$, $\inf_{x \, \in \, \Omega_{G} } K(x) = c_{K} > 0$ and $K\ind_{e} = K(e)$ i.e., it is constant along each edge $e\in E$.
\end{enumerate}

\end{hypothesis}
\begin{remark}\label{Rem Scaling of the Normal Flux}
Notice that the \textit{Hypothesis} \ref{Hyp Forcing Terms and Permeability}-(ii) states that the balance of normal flux on the central vertex is of order $O(n)$ i.e., it scales with the number of incident edges. 
\end{remark} 
\begin{lemma}\label{Th Estimates on v_0 = 0}
Under the Hypothesis \ref{Hyp Forcing Terms and Permeability}, the following facts hold
\begin{enumerate}[(i)]
\item The sequence $\{\pn(0): n\in \N\} \subseteq \R$ is bounded.

\item Let $e\in E$ be an edge of the graph $G$ then, the sequence $\{\deled \pn(0): e\in E_{n}\} \subseteq \R$ is bounded. Moreover, there exists $M_{0}$ such that $\vert \deled \pn (0) \vert \leq M_{0}$ for all $e\in E$ and $n \in \N$ such that $e\in E_{n}$.

\item\label{Hyp Cesaro Means Convergence} Suppose that the sequences $\big\{ \dfrac{1}{n}\sum_{e \in  E_{n}} \int_{0}^{1} (t - 1) F_{e} (t) dt: n\in \N\big\}$, $\big\{\dfrac{1}{n}\, \hn: n\in \N\big\}$ and $\big\{ \dfrac{1}{n}\sum_{e \in E_{n}} K(e): n\in \N\big\}$ are convergent, then the following limits are satisfied
\begin{subequations}\label{Eq Middle Point Behavior}
\begin{equation}\label{Eq Middle Point Evaluation}
\lim_{n\rightarrow \infty} \pn(0) = 
\lim_{n\rightarrow \infty} \frac{\frac{1}{n}\sum_{e \in E_{n}} \int_{0}^{1} (t - 1) F_{e}(t) dt - \frac{1}{n} \hn}{\frac{1}{n}\sum_{e \in E_{n}} K(e)} .
\end{equation}
For any fixed edge $e\in E$ holds
\begin{equation}\label{Eq Middle Point Derivatives Evaluation}
\lim_{n\rightarrow \infty} K(e)\deled\pn(0) = L(e) \defining 
\int_{0}^{1} (t - 1) F_{e}(t)dt 
- K(e)\lim_{n\rightarrow \infty} \frac{\frac{1}{n}\sum_{\sigma \in E_{n}} \int_{0}^{1} (t - 1) F_{\sigma} (t) dt - \frac{1}{n} \hn }{\frac{1}{n}\sum_{\sigma \in E_{n}} K(\sigma)}\, .
\end{equation}
Moreover, the convergence is uniform in the following sense
\begin{equation}\label{Eq Middle Point Derivatives Uniform Convergence}
\forall \, \epsilon > 0 \, \; \exists \, N\in \N\; \text{such that, if}\;n>N \, \text{and }\, e\in E_{n} ,\; \text{then }\; 
\vert K(e) \, \deled \pn(0) - L(e)\vert < \epsilon\, . 
\end{equation}
\end{subequations}
\end{enumerate}
\end{lemma}
\begin{proof}
\begin{enumerate}[(i)]
\item Let $q\in H_{0}^{1}(G_{n})$ be the function such that $q(0) = - 1$ and $q_{e}(t) = t - 1$ for all $e\in E_{n}$. Test \eqref{Pblm n-Stage Weak Variational} with $q$, this yields
\begin{equation}\label{Eq first extraction of middle point}
q(0) \sum_{e \, \in \, E_{n} } K \int_{e}\deled \pn 
= \sum_{e \, \in \, E_{n} }  \int_{e} F q 
+ \hn \,q(0) .
\end{equation}
Computing and doing some estimates we get
\begin{equation*}
\# E_{n} \, c_{K} \, \vert \pn(0) \vert
\leq \frac{1}{\sqrt{3}}
\sum_{e \, \in \, E_{n} }  \Vert F \Vert_{L^{2}(e)}
+  \vert \hn \vert  .
\end{equation*}
Hence
\begin{equation*}
\vert \pn(0) \vert
\leq \frac{1}{c_{K}}
M 
+  \frac{1}{ c_{K}} \,  \Big\vert\frac{ \hn }{n} \Big\vert .
\end{equation*}
This proves the first part.

\item Let $e\in E$ be a fixed edge, let $n\in \N$ be such that $e\in E_{n}$ and let $\pn$ be the solution to \textit{Problem} \eqref{Pblm n-Stage Weak Variational}. Let $q \in H_{0}^{1}(G_{n})$ be as in the previous part and test \eqref{Pblm n-Stage Weak Variational} to get
\begin{equation*}
\begin{split}
- K(e)\pn(0) \, q(0) +  
\sum_{\substack{\sigma \, \in \, E_{n} \\
\sigma \neq e } } K \int_{\sigma} \delsig \pn \, \delsig q
& =  K(e)\, \pn(0)  -  
\sum_{\substack{\sigma \, \in \, E_{n} \\
\sigma \neq e } } K \int_{\sigma}  \delsig^{2} \pn \, q
+ q(0) \sum_{\substack{\sigma \, \in \, E_{n} \\
\sigma \neq e } } K \, \delsig \pn (0)\, 
\\
& = \int_{e} F q 
+ \sum_{\substack{\sigma \, \in \, E_{n} \\
\sigma \neq e } }  \int_{\sigma} F q 
+ \hn \,q(0) .
\end{split}
\end{equation*}
In the expression above, integration by parts was applied to each summand $\sigma\neq e$ of the left hand side, in order to get the second equality. Now, recalling that $\sum_{e\,\in E_{n}} K \,\deled \pn(0) = \hn$ and that $ - K\ind_{e} \deled^{2} \pn = F \ind_{e}$ for each $e\in E_{n}$ the equality above reduces to 
\begin{equation}\label{Eq Relation Function Derivative at the Center}
K(e)\pn(0) = \int_{0}^{1} (t-1) F_{e}(t)dt - K(e)\deled \,\pn(0) .
\end{equation}
Hence,
\begin{equation}\label{Ineq Estimate on the derivative}
\vert  \deled \,\pn(0) \vert \leq 
\vert \pn(0) \vert
+ \frac{1}{c_{K}} \, \Vert  F \Vert_{L^{2}(e)} 
\leq \frac{2}{c_{K}}
M 
+  \frac{1}{c_{K}} \, \Big\vert \frac{\hn}{n} \Big\vert .
\end{equation}
Choosing $M_{0} > 0$ large enough, the result follows. 

\item Let $q \in H_{0}^{1}(G_{n})$ be as in the previous part, testing \eqref{Pblm n-Stage Weak Variational} with it yields the equality \eqref{Eq first extraction of middle point} which is equivalent to
\begin{equation*} 
\pn(0) \, \frac{1}{n}\sum_{e \, \in \, E_{n} } K 
= \frac{1}{n} \sum_{e \, \in \, E_{n} }  \int_{e} F q 
+ \frac{1}{n}\, \hn \,q(0) .
\end{equation*}
Now, letting $n\rightarrow \infty$ the equality \eqref{Eq Middle Point Evaluation} follows because the hypothesis $K(e)> c_{K}$ for all $e\in E$ implies that $\dfrac{1}{n}\sum_{e \, \in \, E_{n}} K(e) \geq c_{K}>0$. For the convergence of $\{\deled \pn(0): e\in E_{n}\}$, let $n\rightarrow \infty$ in the expression \eqref{Eq Relation Function Derivative at the Center} to get the equality \eqref{Eq Middle Point Derivatives Evaluation}. For the uniform convergence observe that the identity \eqref{Eq Relation Function Derivative at the Center} yields
\begin{equation*}
\begin{split}
\big\vert K(e)\deled \pn(0) - L(e) \big\vert 
& = \Big\vert 
K(e) \lim_{n\rightarrow \infty} \frac{\frac{1}{n}\sum_{\sigma\in E_{n}} \int_{0}^{1} (t - 1)F_{\sigma}(t) dt
- \frac{1}{n}  \hn }{\frac{1}{n}\sum_{\sigma \in E_{n}} K(\sigma)} 
- K(e) \pn(0 )\Big\vert \\
& \leq \Vert K\Vert_{L^{\infty}} \Big\vert  \lim_{n\rightarrow \infty} \frac{\frac{1}{n}\sum_{\sigma\in E_{n}} \int_{0}^{1} (t - 1) F_{\sigma}(t) dt - \frac{1}{n} \hn }{\frac{1}{n}\sum_{\sigma\in E_{n}} K(\sigma)} 
- \pn(0 )\Big\vert.
\end{split}
\end{equation*}
Finally, choose $N\in \N$ such that the right hand side of the expression above is less than $\epsilon>0$ for all $n > N$ then, the term of the left hand side is also dominated by $\epsilon > 0$ for all $n > N$ and $e\in E_{n}$.
\qed
\end{enumerate}
\end{proof}
\begin{remark}
It is clear that in \textit{Lemma} \ref{Th Estimates on v_0 = 0} part (iii), it suffices to require the mere existence of the limit
\begin{equation*}
 \lim_{n\rightarrow \infty} \frac{\sum_{\sigma\in E_{n}} \int_{0}^{1} (t - 1) F_{\sigma}(t) dt -  \hn }{\sum_{\sigma\in E_{n}} K(\sigma)}  \, ,
\end{equation*}
in order to attain the same conclusion. 
However, the hypotheses of (iii) are necessary to identify the asymptotic problem and compute the effective coefficients.   
\end{remark}
%
%
%
%
\begin{theorem}\label{Th Bondedness of the function}
Let $F$, $\{\hn: n\in \N\}$ and $K$ verify the Hypothesis \ref{Hyp Forcing Terms and Permeability} as in Lemma \ref{Th Estimates on v_0 = 0} then
\begin{enumerate}[(i)]
\item There exists a constant $M_{1}$ such that $\Vert \pn \Vert_{H^{2}(e)} \leq M_{1}$ for all $e\in E$ and $n\in \N$ such that $e\in E_{n}$.


\item For each $e\in E$ there exists $\psupe\in H^{1}(e)$ such that $\Vert p^{n}\ind_{e} - \psupe \Vert_{H^{1}(e)} \xrightarrow[n \,\rightarrow \,\infty]{} 0$. Moreover, this convergence is uniform in the following sense
\begin{equation}\label{Eq H^1 Uniform Convergence}
\forall \, \epsilon > 0 \, \; \exists \, N\in \N\; \text{such that, if}\;n>N \, \text{and }\, e\in E_{n} ,\; \text{then }\; 
\Vert \deled \pn - \deled \psupe \Vert_{H^{1}(e)} < \epsilon\, . 
\end{equation}

\item The function $p: \Omega_{G}\rightarrow \R$ given by $p \ind_{e} \defining \psupe$ is well-defined and it will be referred to as the \textbf{limit function}. 
\end{enumerate}
\end{theorem}
\begin{proof} 
\begin{enumerate}[(i)]
\item Fix $e\in E$ and let $n\in \N$ be such that $e\in E_{n}$. Since $\pn$ is the solution of \textit{Problem} \eqref{Pblm n-Stage Weak Variational} it follows that $-K(e)\,\deled^{2} \pn = F\ind_{e} \in L^{2}(e)$ for all $e\in E_{n}$, in particular $\pn\ind_{e} \in H^{2}(e)$ with $\Vert \deled^{2} \pn \Vert_{L^{2}(e)} \leq \dfrac{1}{c_{K}} \, \Vert F \Vert_{L^{2}(e)} \leq \dfrac{1}{c_{K}}  M$. On the other hand, since $\deled \pn \,\ind_{e}$ is absolutely continuous, the fundamental theorem of calculus applies, hence $\deled \pn(x)  = \deled \pn(0) + \int_{0}^{x}\del^{2} \pne (t)\, dt = \deled \pn(0) + \int_{0}^{x} F_{e}(t)\, dt$ for all $x\in e$. Therefore,
\begin{equation*}
\vert \deled \pn(x) \vert^{2} = 2 \vert \deled \pn(0)\vert^{2} 
+ 2 \,x \,\Vert F\Vert_{L^{2}(e)}^{2} \leq 2\, M_{0}^{2} + 2 M^{2} .
\end{equation*}
Where $M_{0}$ is the global bound found in \textit{L}emma \ref{Th Estimates on v_0 = 0}-(ii) above. 
Integrating along the edge $e$ gives $\Vert \deled \pn \Vert_{L^{2}(e)} \leq  \sqrt{2(M_{0}^{2}+ M^{2})}$. Next, given that $\pn(v) = 0$ for all $v\in E_{n}$, repeating the previous argument yields $\Vert  \pn \Vert_{L^{2}(e)} \leq  \sqrt{2(M_{0}^{2}+ M^{2})}$. Finally, since $\Vert \deled^{2} \pn \Vert_{L^{2}}(e) \leq \dfrac{1}{c_{K}}  M$, the result follows for any $M_{1}$ satisfying
\begin{equation*}
M_{1}^{2} \geq 4M_{0}^{2}  + \Big(4 + \frac{1}{c_{K}^{2}} \Big)  M^{2} .
\end{equation*}

\item Fix $e\in E$, due to the previous part the sequence $\{\pn\ind_{e}: e\in E_{n} \}$ is bounded in $H^{2}(e)$, then there exists $\psupe\in H^{2}(e)$ and a subsequence $\{n_{k}: k\in \N \}$ such that 
\begin{equation*}
p^{n_{k}} \xrightarrow[k\,\rightarrow \,\infty]{} \psupe \quad \text{weakly in}\, H^{2}(e) \; \text{and strongly in}\, H^{1}(e).
\end{equation*}
Let $\varphi \in H^{1}(e)$ such that equals zero on the boundary vertex of $e$. Let $q$ be the function in $H_{0}^{1}(G_{n})$ such that $q_{e} = \varphi$ and $q_{\sigma} (t)=\varphi(0) (1 - t)$ is linear for all $\sigma\in E_{n} -\{e\}$. Test \textit{Problem} \eqref{Pblm n-Stage Weak Variational} with this function to get
\begin{equation*}
\int_{e}K(e)\deled \pn \, \deled q+  
\sum_{\substack{\sigma \, \in \, E_{n} \\
\sigma \neq e } }  K \int_{\sigma} \delsig \pn \, \delsig q
= \int_{e} F q 
+ \sum_{\substack{\sigma \, \in \, E_{n} \\
\sigma \neq e } }  \int_{\sigma} F q 
+ \hn \,q(0) .
\end{equation*}
Integrating by parts the second summand of the left hand side yields
\begin{equation*}
\int_{e}K(e)\deled \pn \, \deled \varphi -  
\sum_{\substack{\sigma \, \in \, E_{n} \\
\sigma \neq e } }  K \int_{\sigma} \delsig^{2} \pn \,  q
- \varphi(0)\sum_{\substack{\sigma \, \in \, E_{n} \\
\sigma \neq e } }  K \, \delsig \pn (0) 
= \int_{e} F \varphi 
+ \sum_{\substack{\sigma \, \in \, E_{n} \\
\sigma \neq e } }  \int_{\sigma} F \varphi 
+ \hn \,q(0) .
\end{equation*}
Since $\pn$ is a solution of the problem, the above reduces to
\begin{equation}\label{Stmt Isolated Edge Variational Statement n-Stage}
\int_{e}K(e)\deled \pn \, \deled \varphi 
+  K(e)\deled \pn (0)\, \varphi(0)
= \int_{e} F q  .
\end{equation}
Equality \eqref{Stmt Isolated Edge Variational Statement n-Stage} holds for all $n\in \N$, in particular it holds for the convergent subsequence $\{n_{k}: k\in \N\}$, taking limit on this sequence and recalling \eqref{Eq Middle Point Derivatives Evaluation}, we have
\begin{equation}\label{Stmt Isolated Edge Variational Statement Limit}
\int_{e}K(e)\deled \psupe \, \deled \varphi 
= \int_{e} F \varphi  
- L(e) \, \varphi(0).
\end{equation}
The statement \eqref{Stmt Isolated Edge Variational Statement Limit} holds for all $\varphi\in H^{1}(e)$ vanishing at $v_{e}$, the boundary vertex of $e$. Define the space $H(e)\defining \{ \varphi \in H^{1}(e): \varphi(v_{e}) = 0\}$ and consider the problem
\begin{align}\label{Eq Limit Problem on the Edge}
& u\in H(e): &
\int_{e}K(e)\deled u \, \deled \varphi  = \int_{e} F  \varphi  
- L(e) \, \varphi(0) \, ,&
& \forall \,\varphi \in H(e) .
\end{align}
Due to the Lax-Milgram Theorem the problem above is well-posed, additionally it is clear that $\psupe\in H(e)$, therefore it is the unique solution to the \textit{Problem} \eqref{Eq Limit Problem on the Edge} above. Now, recall that $\{\pn\ind_{e}:  e\in E_{n}\}$ is bounded in $H^{2}(e)$ and that the previous reasoning applies for every strongly $H^{1}(e)$-convergent subsequence, therefore its limit is the unique solution to \textit{Problem} \eqref{Eq Limit Problem on the Edge}. Consequently, due to Rellich-Kondrachov, it follows that the whole sequence converges strongly. Next, for the uniform convergence test both statements \eqref{Stmt Isolated Edge Variational Statement n-Stage} and \eqref{Stmt Isolated Edge Variational Statement Limit} with $(\pn\ind_{e} - \psupe )$ and subtract them to get 
\begin{equation*} 
\begin{split}
c_{K}\Vert \deled \pn - \deled \psupe \Vert^{2}_{H^{1}(e)} \leq 
K(e)\int_{e}\big\vert \deled \pn - \deled \psupe \big\vert^{2} & =
\big(  L(e) - K(e)\deled \pn (0) \big) \big(\pn(0) - \psupe (0) \big) \\
& \leq \big\vert L(e) - K(e)\deled \pn (0) \big\vert \, \Vert \deled \pn - \deled \psupe \Vert_{H^{1}(e)} .
\end{split}
\end{equation*}
The above yields
\begin{equation*} 
\Vert \deled \pn - \deled \psupe \Vert_{H^{1}(e)} 
\leq \frac{1}{c_{K}} \,\big\vert  L(e) - K(e)\deled \pn (0) \big\vert .
\end{equation*}
Now, the uniform convergence \eqref{Eq H^1 Uniform Convergence} follows from the \textit{Statement} \eqref{Eq Middle Point Derivatives Uniform Convergence}, which concludes the second part. 

\item Since $\psupe (0) = \lim\limits_{n\to \infty} \pn(0)$ for all $e\in E$ then, the limit function $p$ is well-defined and the proof is complete.
\qed
\end{enumerate}
\end{proof}
\section{The Homogenized Problem: a Ces\`{a}ro Average Approach}\label{Sec Radial Approach}
%
%
%
%
In this section we study the asymptotic properties of the global behavior of the solutions $\{\pn:n\in \N \}$. It will be seen that such analysis must be done for certain type of ``Ces\`{a}ro averages" of the solutions. This is observed by the techniques and the hypotheses  of \textit{Lemma} \ref{Th Estimates on v_0 = 0}, which are necessary to conclude the local convergence of $\{\pn\ind_{e}: e \in E_{n} \}$. Additionally, the type of estimates and the numerical experiments suggest this physical magnitude as the most significant for global behavior analysis and upscaling purposes. We start introducing some necessary hypotheses.
%
%
%
%
%
%
%
%
%
%
%
%
\begin{hypothesis}\label{Hyp Forcing Term Cesaro Means Assumptions}
Suppose that $F$, $\{\hn: n\in \N\}$ and $K$ verify \textit{Hypothesis} \ref{Hyp Forcing Terms and Permeability} and, additionally 
\begin{enumerate}[(i)]
\item The diffusion coefficient $K: \Omega_{G}\rightarrow (0, \infty)$ has finite range. Moreover, if $K(E) = \{K_{i}: 1\leq  i \leq I \}$ and $B_{i} \defining \{e\in E: K(e) = K_{i}\}$, then 
\begin{equation}
\frac{1}{n}\sum_{e\, \in \, E_{n} \cap B_{i}} K(e) 
= \frac{\# (E_{n}\cap B_{i})}{n} K_{i}
\xrightarrow[n\,\rightarrow \,\infty]{} s_{i} \, K_{i}.
\end{equation}
With $s_{i} > 0$ for all $1\leq i\leq I$ and such that $\sum\limits_{i = 1}^{I} s_{i} = 1$.

\item The forcing term $F$ satisfies that
\begin{align}\label{Eq Forcing Term Cesaro Means Assumptions}
& \frac{1}{\#(E_{n}\cap B_{i})}\sum_{e\,\in \,E_{n} \cap B_{i}} F_{e} 
\xrightarrow[n\,\rightarrow \, \infty]{} \Fi \, , &
& \forall \; 1\leq i\leq I.
\end{align}
Where $\Fi\in L^{2}(0,1)$ and the sense of convergence is pointwise almost everywhere. 

\item The sequence $\big\{\dfrac{1}{n} \, \hn: n\in \N \big\}$ is convergent with $\h = \lim\limits_{n\rightarrow \infty} \dfrac{1}{n}\, \hn$.
\end{enumerate}
\end{hypothesis}
\begin{remark}\label{Rem Riemann Integrability function}
\begin{enumerate}[(i)]
\item Notice that if (i) and (ii) in \textit{Hypothesis} \ref{Hyp Forcing Term Cesaro Means Assumptions} are satisfied, then
\begin{equation*} 
\frac{1}{n}\sum_{e\,\in\, E_{n}} F_{e} = 
 \sum_{i\, =\, 1}^{I} \frac{\#(E_{n}\cap B_{i})}{n}
 \frac{1}{\#(E_{n}\cap B_{i})}\sum_{e\,\in \,E_{n} \cap B_{i}} F_{e} 
\xrightarrow[n\,\rightarrow \, \infty]{} \sum_{i\, =\, 1}^{I}s_{i}\Fi \, .
\end{equation*}
Hence, the sequence $\{F_{e}: e\in E \}$ is Ces\`{a}ro convergent.

\item 
A familiar context for the required convergence statement \eqref{Eq Forcing Term Cesaro Means Assumptions} in \textit{Hypothesis} \ref{Hyp Forcing Term Cesaro Means Assumptions} above is the following. Let $F$ be a continuous and bounded function defined on the whole disk $\Omega$ and suppose that for each $1\leq i \leq I$, the sequence of vertices $\{v_{n}: n\in \N, \, v_{n}v_{0} \in B_{i} \}$ is equidistributed on $S^{1}$. Then, due to \textit{Weyl's Theorem} \ref{Th Weyl's Theorem}, for any fixed $t\in (0,1)$ it holds that $\dfrac{1}{\#(E_{n} \cap B_{i})}\sum_{e\,\in\,E_{n} \cap B_{i}} F_{e}(t) \xrightarrow[n\,\rightarrow \, \infty]{}
\mrad[f]$ i.e, the angular average introduced in \textit{Definition} \ref{Def angular average}.   
\end{enumerate}
\end{remark}
%
%
%
%
%
\subsection{Estimates and Ces\`aro Convergence Statements}\label{Sec Estimates and Cesaro Convergence}
%
%
%
%
%
\begin{lemma}\label{Th Boundedness on Each Choice}
Let $F$, $\{\hn:n\in \N \}$ and $K$ verify Hypothesis \ref{Hyp Forcing Term Cesaro Means Assumptions} then
\begin{enumerate}[(i)]
\item The sequence $\big\{\dfrac{1}{n}\sum_{e\,\in\, E_{n}} \pn_{e} : n\in \N \big\}$ is bounded in $H^{2}(0,1)$. 

\item The sequence 
\begin{equation} 
\Big\{\frac{1}{\#(E_{n} \cap B_{i})}\sum_{e\,\in\, E_{n} \cap B_{i}} \pn_{e} : n\in \N \Big\}
\end{equation}
is bounded in $H^{2}(0,1)$ for all $i\in \{ 1, \ldots, I\}$.
\end{enumerate}
\end{lemma}
\begin{proof}
\begin{enumerate}[(i)]
\item Test \textit{Problem} \eqref{Pblm n-Stage Weak Variational} with $\pn$ to get
\begin{equation*} 
\begin{split}
c_{K} \sum_{e \, \in \, E_{n}} \Vert \deled \pn \Vert_{L^{2}(e)}^{2} \leq 
\sum_{e \, \in \, E_{n}} \int_{e} K \vert \del\pn \vert^{2} 
& = 
\sum_{e \, \in \, E_{n}} \int_{e} F_{e}\,  \pn 
+ \hn \,  \pn (v_{0}) \\
& \leq \Big(\sum_{e \, \in \, E_{n}} \Vert  F \Vert_{L^{2}(e)}^{2}\Big)^{1/2}
\Big(\sum_{e \, \in \, E_{n}} \Vert  \pn \Vert_{L^{2}(e)}^{2}\Big)^{1/2}
+ \vert \hn \vert \,  \vert \pn (v_{0})\vert .
\end{split} 
\end{equation*}
Since $\pn(v_{e}) = 0$ for all $e\in E_{n}$, then $\Vert \pn \Vert_{L^{2}(e)}\leq \Vert \del\pn \Vert_{L^{2}(e)}$ and $\Vert \pn \Vert_{H^{1}(e)} \leq \sqrt{2} \, \Vert \del\pn \Vert_{L^{2}(e)}$. Hence, dividing the above expression over $n$ gives
\begin{equation*} 
\begin{split}
\frac{1}{n} \sum_{e \, \in \, E_{n}} \Vert \pn \Vert_{H^{1}(e)}^{2}
& \leq 2 \, \Big(\frac{1}{n}\sum_{e \, \in \, E_{n}} \Vert  F \Vert_{L^{2}(e)}^{2}\Big)^{1/2}
\Big(\frac{1}{n}\sum_{e \, \in \, E_{n}} \Vert   \pn \Vert_{H^{1}(e)}^{2}\Big)^{1/2}
+ 2\, \frac{\vert \hn \vert}{n} \,  \vert \pn(v_{0}) \vert \\
& \leq 2\, \frac{M}{c_{K}} \Big(\frac{1}{n}\sum_{e \, \in \, E_{n}} \Vert  \pn \Vert_{H^{1}(e)}^{2}\Big)^{1/2} 
+ C \, .
\end{split}
\end{equation*}
Here $C> 0$ is a generic constant independent from $n\in \N$. 
In the second line of the expression above, we used that $M = \sup_{e\in E_{n}} \Vert F \Vert_{L^{2}(e)} < +\infty$, $\{\frac{1}{n}\, \hn: n\in \N \}$ are bounded and that $\{\pn(v_{0}): n\in \N \}$ is convergent (therefore bounded) as stated in \textit{Lemma} \ref{Th Estimates on v_0 = 0}-(i). Hence, the sequence $x_{n} \defining (\frac{1}{n} \sum_{e \, \in \, E_{n}} \Vert  \pn \Vert_{H^{1}(e)}^{2} )^{1/2}$ is such that $x_{n}^{2} \leq 2\, \frac{M}{c_{K}}\, x_{n} + C$ for all $n\in \N$, where the constants are all non-negative. Then $\{x_{n}: n\in \N \}$ must be bounded, but this implies
\begin{equation*} 
\Big\Vert \frac{1}{n} \sum_{e\, \in \, E_{n}} \pne \Big\Vert_{H^{1}(0,1)}
\leq \frac{1}{n}\sum_{e \, \in \, E_{n}} \Vert  \pne \Vert_{H^{1}(0,1)}
=  \frac{1}{n}\sum_{e \, \in \, E_{n}} \Vert  \pn \Vert_{H^{1}(e)}
\leq \Big(\frac{1}{n}\sum_{e \, \in \, E_{n}} \Vert  \pn \Vert_{H^{1}(e)}^{2}\Big)^{1/2} \, .
\end{equation*}
Finally, recalling the estimate 
\begin{equation*} 
c_{K}\Big\Vert \frac{1}{n} \sum_{e\, \in \, E_{n}} \del^{2}\pne \Big\Vert_{L^{2}(0,1)}
\leq  \frac{1}{n} \sum_{e\, \in \, E_{n}} \big\Vert \del K(e)\del \pne \big\Vert_{L^{2}(0,1)}
=  \frac{1}{n} \sum_{e\, \in \, E_{n}} \Vert F \Vert_{L^{2}(0,1)}
\leq M ,
\end{equation*}
the result follows. 

\item Fix $i\in \{1, 2, \ldots, I\}$ then
\begin{equation*} 
\Big\Vert\frac{1}{\#(E_{n} \cap B_{i})}\sum_{e\,\in\, E_{n} \cap B_{i}} \pn_{e}  \Big\Vert_{H^{2}(0,1)}
\leq  \frac{n}{\#(E_{n} \cap B_{i})}\, \frac{1}{n} \sum_{e\, \in \, E_{n}} \big\Vert  \pne \big\Vert_{H^{2}(0,1)} .
\end{equation*}
On the right hand side term of the expression above, the first factor is bounded due toHypothesis \ref{Hyp Forcing Term Cesaro Means Assumptions}- (iii), while the boundedness of the second factor was shown in the previous part. Therefore, the result follows.
\qed
\end{enumerate}
\end{proof}
Before presenting the limit problem we introduce some necessary definitions and notation
\begin{definition}\label{Del Limit Graph and Embedding Technique}
Let $K$ verify \textit{Hypothesis} \ref{Hyp Forcing Term Cesaro Means Assumptions} and $I = \#K(E)$ then
\begin{enumerate}[(i)]
\item For all $1\leq i \leq I$ define $w_{i} \defining \big(\cos (\frac{2\pi}{I}\, i), \sin (\frac{2\pi}{I}\, i) \big) \in S^{1}$, $w_{0} \defining v_{0} = 0$ and $\V \defining \{w_{i}: 0\leq i\leq I \} $.

\item For all $1\leq i \leq I$ define the edges $\sigma_{i} \defining w_{0}w_{i}$ and $\E \defining \{\sigma_{i}: 1\leq i\leq I \}$.

\item Define the \textbf{upscaled graph} by $\G \defining (\V, \E )$.

\item For any $\varphi\in H_{0}^{1}(\G)$ and $n\in \N$ denote by $T_{n}\varphi\in H_{0}^{1}(G_{n})$, the function such that $\big(T_{n}\varphi\big) \ind_{e}$ agrees with $\varphi \ind_{\sigma_{i}}$ whenever $e\in B_{i}$. This is summarized in the expression
\begin{equation*}
T_{n}\varphi = \sum_{e\,\in \, E_{n}} \big(T_{n}\varphi\big) \ind_{e} 
\defining 
\sum_{i \, = \, 1}^{I} \; 
\sum_{e\, \in  \, E_{n} \cap B_{i}}  \big(\varphi \ind_{\sigma_{i}}\big)  \, .
\end{equation*}
In the sequel we refer to $T_{n} \varphi$ as the $\bm{H_{0}^{1}(G_{n})}$\textbf{-embedding} of $\varphi$.

\item In the following, for any $1\leq i \leq I$, we adopt the notation $\deli \defining \del_{\sigma_{i}}$. Similarly, for any given function $f: \Omega_{\G}\rightarrow \R$ and edge $\sigma_{i} \in \E$ we denote by $f_{i}: (0,1)\rightarrow \R$, the real variable function $f_{i}(t) \defining (f\ind_{\sigma_{i}})\big(t \cos (\frac{2\pi}{I} i), t \sin (\frac{2\pi}{I} i)\big)$.
\end{enumerate}
\end{definition}
\begin{theorem}\label{Th The Upscaled Problem}
Let $F$, $\{\hn:n\in \N \}$ and $K$ verify Hypothesis \ref{Hyp Forcing Term Cesaro Means Assumptions} then
\begin{enumerate}[(i)]
\item The following problem is well-posed.
\begin{align}\label{Stmt limit problem}
& \p \in H_{0}^{1}(\G) : &
&  \sum_{i \, = \, 1}^{I} \int_{\sigma_{i}} s_{i} K_{i} \, \deli \p \; \deli q = 
\sum_{i \, = \, 1}^{I}  \int_{0}^{1} s_{i} \, \F_{i}\,  q_{i} 
+ \h(0)\,  q(0) \, , &
& \forall \, q\in H_{0}^{1}(\G) .
\end{align}
In the sequel, we refer to \textit{Problem} \eqref{Stmt limit problem} as the \textbf{upscaled} or the \textbf{homogenized problem} and its solution $\p$ as the \textbf{upscaled} or the \textbf{homogenized solution} indistinctly. 

\item The sequence of solutions $\{\pn:n\in \N \}$ satisfy
\begin{align}\label{Stmt Strong Convergence to Limit Solutions}
& \Big\Vert \frac{1}{\# (E_{n}\cap B_{i})}\sum_{e \, \in \, E_{n} \cap B_{i} }\pne - \plimi\, \Big\Vert_{H^{1}(0,1)}\xrightarrow[n\,\rightarrow \, \infty]{} 0 \, , &
& \forall\; 1\leq i\leq I \, .
\end{align}

\item The limit function $p:\Omega_{G}\rightarrow \R$ satisfies
\begin{align*}
& \Big\Vert \frac{1}{\# (E_{n}\cap B_{i})} \sum_{e \, \in \, E_{n} \cap B_{i}}\pe - \plimi\, \Big\Vert_{H^{1}(0,1)}\xrightarrow[n\,\rightarrow \, \infty]{} 0 \, , &
&  \forall\; 1\leq i\leq I \, .
\end{align*}
\begin{equation*}
\Big\Vert \frac{1}{n} \sum_{e \, \in \, E_{n} }\pe - \sum_{i\, =\, 1}^{I} s_{i} \, \plimi\, \Big\Vert_{H^{1}(0,1)}\xrightarrow[n\,\rightarrow \, \infty]{} 0 \, .
\end{equation*}
\end{enumerate}
\end{theorem}
\begin{proof}
\begin{enumerate}[(i)]
\item It follows immediately from \textit{Proposition} \ref{Th Well-Posedness Weak variational on Graphs}.

\item Let $\varphi \in H_{0}^{1}(\G)$ and let $T_{n}\varphi$ be its $H_{0}^{1}(G_{n})$-embedding. Notice the equalities 
\begin{subequations}\label{Eq associations by color}
\begin{equation}\label{Eq associations by color of solution}
\begin{split}
\frac{1}{n}\sum_{e \, \in \, E_{n}} \int_{e} K \deled \pn_{e}\, \deled T_{n}\varphi 
& = \sum_{i\, = \, 1}^{I}\frac{1}{n} \, K_{i}\sum_{e \, \in \, E_{n} \cap B_{i}} \int_{e} \deled \pn\, \deled T_{n} \varphi \\
& = \sum_{i\, = \, 1}^{I}\frac{\# (E_{n} \cap B_{i})}{n} \, K_{i}
\int_{0}^{1} \, \del \Big( \frac{1}{\# (E_{n} \cap B_{i})}\sum_{e \, \in \, E_{n} \cap B_{i}} \pne \Big)\, \del \varphi_{i} \, ,
\end{split}
\end{equation}
and
\begin{equation}\label{Eq associations by color of forcing term}
%
\frac{1}{n}\sum_{e \, \in \, E_{n}} \int_{e} F\,  T_{n} \varphi
%
= \sum_{i\, = \, 1}^{I}\frac{\# (E_{n} \cap B_{i})}{n} 
\int_{0}^{1} \,  \Big( \frac{1}{\# (E_{n} \cap B_{i})}\sum_{e \, \in \, E_{n} \cap B_{i}} F_{} \Big)  \varphi_{i} . 
%
%
\end{equation}
\end{subequations}
Now, testing \textit{Problem} \eqref{Pblm n-Stage Weak Variational} with $\dfrac{1}{n}\, T_{n}\varphi$, due to the previous observations gives
\begin{multline}\label{Eq Testing on the n-stage}
\sum_{i\, = \, 1}^{I}\frac{\# (E_{n} \cap B_{i})}{n} \, K_{i}
\int_{0}^{1} \, \del \Big( \frac{1}{\# (E_{n} \cap B_{i})}\sum_{e \, \in \, E_{n} \cap B_{i}} \pne \Big)\, \del \varphi_{i} \\
=
\sum_{i\, = \, 1}^{I}\frac{\# (E_{n} \cap B_{i})}{n} 
\int_{0}^{1} \,  \Big( \frac{1}{\# (E_{n} \cap B_{i})}\sum_{e \, \in \, E_{n} \cap B_{i}} F_{} \Big)  \varphi_{i} 
+ \frac{\hn}{n}\, \varphi(0).
\end{multline}
Due to \textit{Lemma} \ref{Th Boundedness on Each Choice}-(ii) there must exist a subsequence $\{n_{k}:k\in \N \}$ and a collection $\{\xi_{i}: 1\leq i\leq I\} \subseteq H^{2}(0,1)$ such that 
\begin{equation*}
\frac{1}{\# (E_{n_{k}} \cap B_{i})}\sum_{e \, \in \, E_{n_{k}} \cap B_{i}} p^{n_{k}}_{e}  \xrightarrow[k\,\rightarrow\, \infty]{} \xi_{i} \quad \text{weakly in}\; H^{2}(0,1)\, \text{and strongly in} \; H^{1}(0,1)  , \,  \text{for all}\, 1\leq i \leq I .
\end{equation*}
On the other hand, due to \textit{Hypothesis} \ref{Hyp Forcing Term Cesaro Means Assumptions}-(ii) the integrand of the right hand side in the identity \eqref{Eq Testing on the n-stage} is convergent for all $i\in \{1,\ldots, I \}$. Due to \textit{Hypothesis} \ref{Hyp Forcing Term Cesaro Means Assumptions}-(i) the sequences $\{ \frac{\#(E_{n} \cap B_{i})}{n}: n\in \N\}$ are also convergent for all $i\in \{1, \ldots, I \}$. Then, taking the equality \eqref{Eq Testing on the n-stage} for $n_{k}$ and letting $k\rightarrow \infty$ gives
\begin{equation*} 
\sum_{i\, = \, 1}^{I}s_{i}\, K_{i}
\int_{0}^{1} \, \del \xi_{i}\, \del \varphi_{i} 
=
\sum_{i\, = \, 1}^{I}s_{i}
\int_{0}^{1} \,  \Fi \,  \varphi_{i} 
+ \h\, \varphi(0).
\end{equation*}
Notice that since $p^{n_{k}}\in H_{0}^{1}(G_{n_{k}})$ then 
\begin{equation}\label{Eq limit statement}
\frac{1}{\# (E_{n_{k}} \cap B_{i})}\sum_{e \, \in \, E_{n_{k}} \cap B_{i}} p^{n_{k}}_{e} (0)
= \frac{1}{\# (E_{n_{k}} \cap B_{j})}\sum_{e \, \in \, E_{n_{k}} \cap B_{j}} p^{n_{k}}_{e} (0) \,, \quad
\forall \, i, j \in\{1, \ldots, I\} .
\end{equation}
In particular $\xi_{i}(0) = \xi_{j}(0)$; consequently the function $\eta\in H_{0}^{1}(\G)$ such that $\eta_{i} = \xi_{i}$ is well-defined. Moreover, the identity \eqref{Eq limit statement} above is equivalent to 
\begin{equation*} 
\sum_{i \, = \, 1}^{I} \int_{\sigma_{i}} s_{i} K_{i} \, \deli \eta \; \deli \varphi = 
\sum_{i \, = \, 1}^{I}  \int_{\sigma_{i}} s_{i} \Fi\,  \varphi
+ \h(0)\,  \varphi(0) \, .
\end{equation*}
Since the variational statement above holds for any $\varphi \in H_{0}^{1}(\G)$ and $\eta \in H_{0}^{1}(\G)$, due to the previous part it follows that $\eta \equiv \p$. Finally, the whole sequence $\{\pn: n\in \N\}$ satisfies \eqref{Stmt Strong Convergence to Limit Solutions} because, for every subsequence $\{p^{n_{j}}: j\in \N \}$ there exists yet another subsequence $\{p^{n_{j_{\ell}}}: \ell\in \N \}$ satisfying the convergence \textit{Statement} \eqref{Stmt Strong Convergence to Limit Solutions}. This concludes the second part. 

\item Both conclusions follow immediately from the previous part and the uniform convergence \textit{Statement} \eqref{Eq H^1 Uniform Convergence} shown in \textit{Theorem} \ref{Th Bondedness of the function}.
\qed
\end{enumerate}
\end{proof}
\begin{remark}[Probabilistic Flexibilities of the Results]\label{Rem Probabilistic Flexibilities}
Consider the following random variables
\begin{enumerate}[(i)]
\item Let $X: E\rightarrow (0, \infty)$ be a random variable of finite range $\{K_{i}:1\leq i \leq I\}$ and such that $\Exp[X = K_{i}] = s_{i}$ for all $1\leq i \leq I$. Notice that due to the Law of Large Numbers, with probability one it holds that
\begin{equation}\label{Eq Probabilistic Diffusion Coefficient of Finite Range}
\frac{1}{n}\sum_{e\, \in \, E_{n} \cap B_{i}} X(e) \xrightarrow[n\,\rightarrow \,\infty]{} s_{i} \, K_{i}.
\end{equation}

\item Let $Y: E\rightarrow L^{2}(0,1)$ be a random variable such that $\sup_{e\,\in \, E}\Vert Y(e) \Vert_{L^{2}(e)} < +\infty$ and  such that
\begin{align}\label{Eq Random Forcing Vector}
& \frac{1}{\#(E_{n}\cap B_{i})}\sum_{e\, \in \, E_{n} \cap B_{i}} Y(e) \xrightarrow[n\,\rightarrow \,\infty]{} \F_{i} \, ,&
& \forall \; 1\leq i \leq I.
\end{align}
\end{enumerate}
Therefore, the results of \textit{Theorem} \ref{Th The Upscaled Problem} hold, when replacing $K$ by $X$ or $F$ by $Y$ or when making both substitutions at the same time.
\end{remark}
\section{The Examples}\label{Sec Numerical Experiments}
%
%
%
%
In this section we present two types of numerical experiments. The first type are verification examples, supporting our homogenization conclusions for a problem whose asymptotic behavior is known exactly.  The second type are of exploratory nature, in order to gain further understanding of the phenomenon's upscaled behavior. The experiments are executed in a MATLAB code using the Finite Element Method (FEM); it is an adaptation of the code \textbf{fem1d.m} \cite{PeszynskaFEM}.
%
%
\subsection{General Setting}
%
%
For the sake of simplicity the vertices of the graph are given by $v_{\ell} \defining (\cos \ell, \sin \ell)\in S^{1}$, as it is known that $\{v_{\ell}: \ell\in \N \}$ is equidistributed in $S^{1}$ (see \cite{SteinShakarchi}). The diffusion coefficient hits only two possible values one and two. Two types of coefficients will be analyzed, $K_{d}, K_{p}$ a deterministic and a probabilistic one respectively. They satisfy 
\begin{subequations}\label{Def Experimental Difussion Coefficients}
\begin{align}\label{Def Deterministic Experimental Difussion Coefficient}
& K_{d}: \Omega_{G}\rightarrow \{1, 2 \} \, , & 
& K_{d}(v_{\ell}v_{0})  \defining \begin{cases}
1  , & \ell \equiv 0 \mod 3, \\
2 , & \ell \not\equiv 0 \mod 3.
\end{cases} 
\end{align}
\begin{align}\label{Def Probabilistic Experimental Difussion Coefficient}
& K_{p}: \Omega_{G}\rightarrow \{1, 2 \} \, , &
&  \Exp[K_{p} = 1] = \frac{1}{3} \, , \;  \Exp[K_{p} = 2] = \frac{2}{3} .
\end{align} 
\end{subequations}
In our experiments the asymptotic analysis is performed for $K_{p}$ being a fixed realization of a random sequence of length 1000, generated with the binomial distribution $1/3, 2/3$. Since $\# K_{d}(E) = \# K_{p}(E) = 2$ it follows that the upscaled graph $\G$ has only three vertices and two edges namely $w_{1} = (1, 0)$, $w_{2} = (-1, 0)$, $w_{0} = (0, 0)$ and $\sigma_{1} = w_{1}w_{0} $, $\sigma_{2} = w_{2}w_{0} $. Also, define the domains 
\begin{align*}
& \Omega^{1}_{G} \defining \bigcup \big\{v_{\ell}v_{0}: \ell\in\N\, , \;K(v_{\ell}v_{0}) = 1\big\} \, , &
& \Omega^{2}_{G} \defining \bigcup \big\{v_{\ell}v_{0}: \ell\in\N\, , \; K(v_{\ell}v_{0}) = 2 \big\} \, .
\end{align*}
Where $K = K_{d}$ or $K = K_{p}$ depending on the probabilistic or deterministic context. Additionally, we define
\begin{align*}
& \phone \defining \frac{1}{\#(E_{n} \cap B_{1})} \sum_{e\,\in\, E_{n} \cap B_{1}} \pne &
& \phtwo \defining \frac{1}{\#(E_{n} \cap B_{2})} \sum_{e\,\in\, E_{n} \cap B_{2}} \pne \, .
\end{align*}
For all the examples we use the forcing terms $\hn = 0$ for every $n\in \N$. The FEM approximation is done with $100$ elements per edge with uniform grid. For each example we present two graphics for values of $n$ chosen from $\{10, 20, 50, 100, 500, 1000\}$, based on optical neatness. For visual purposes in all the cases the edges are colored with red if $K(e) = 1$ or blue if $K(e) = 2$. Also, for displaying purposes, in the cases $n \in \{10, 20\}$ the edges $v_{\ell}v_{0}$ are labeled with ``$\ell\,$" for identification, however for $n\in\{50, 100, 500, 1000\}$ the labels were removed because they overload the image. 
%
%
%
\subsection{Verification Examples}
%
%
%
%
\begin{example}[A Riemann Integrable Forcing Term]\label{Ex. Basic Example}
We begin our examples with the most familiar context as discussed in \textit{Remark} \ref{Rem Riemann Integrability function}. Define $F: \Omega\rightarrow \R$ by $F(t \cos \ell, t\sin \ell)\defining \pi^{2}\,  \sin (\pi t) \cos(\ell)$. Since both sequences $\{v_{\ell}: \ell \in \N, \ell \equiv 0 \mod 3 \}$ and $\{v_{\ell}: \ell \in \N, \ell \not\equiv 0 \mod 3 \}$ are equidistributed, \textit{Weyl's Theorem} \ref{Th Weyl's Theorem} implies 
\begin{equation*}
\F_{1} = 
\mrad[F\vert_{\Omega_{G}^{1}}]=
\F_{2} = 
\mrad[F\vert_{\Omega_{G}^{2}}] = 
\mrad[F] \equiv 0 .
\end{equation*}
Here $\F_{1}$, $\F_{2}$ are the limits defined in \textit{Hypothesis} \ref{Hyp Forcing Term Cesaro Means Assumptions}-(ii). For this case the exact solution of the upscaled \textit{Problem} \eqref{Stmt limit problem} is given by $\p \defining \p \, \ind_{\sigma_{1}} + \p \, \ind_{\sigma_{2}} \in H_{0}^{1}(\G)$, with $\p_{1}(t) = \p_{2}(t) = 0$. For the diffusion coefficient we use the deterministic one, $K_{d}$ defined in \eqref{Def Deterministic Experimental Difussion Coefficient}. The following table summarizes the convergence behavior.
%
%
\begin{table*}[h!]
\centerline{\textit{Example} \ref{Ex. Basic Example} : Convergence Table, $K = K_{d}$.}\vspace{5pt}
\def\arraystretch{1.4}
\begin{center}
\begin{tabular}{ c c c c c }
    \hline
$n$  & $\Vert  \phone- \p_{1}  \Vert_{L^{2}(_{e_{1}})} $ & $\Vert  \phtwo- \p_{2}  \Vert_{L^{2}(e_{1})}$ & $\Vert  \phone- \p_{1}  \Vert_{H^{1}_{0}(e_{2})}$ & $\Vert  \phtwo- \p_{2}  \Vert_{H^{1}_{0}(e_{2})}$ \\ 
    \toprule
10 &  0.3526 & 0.1717 & 0.8232 & 0.3216  \\
20 & 0.0180  & 0.0448 & 0.0900 &  0.0889     \\
100 & 0.0160  & 0.0059 & 0.0395 &  0.0116    \\
1000  & $5.8352\times 10^{-4}$  & $8.27772\times 10^{-4}$ & 0.0012 &  0.0016   \\ 
    \hline
\end{tabular}
\end{center}
\end{table*}
\begin{figure}[h] 
\vspace{-1.3cm}
        \centering
        \begin{subfigure}[Solution $p^{10}$, $n = 10$.]
                {\resizebox{7cm}{7cm}
                {\includegraphics{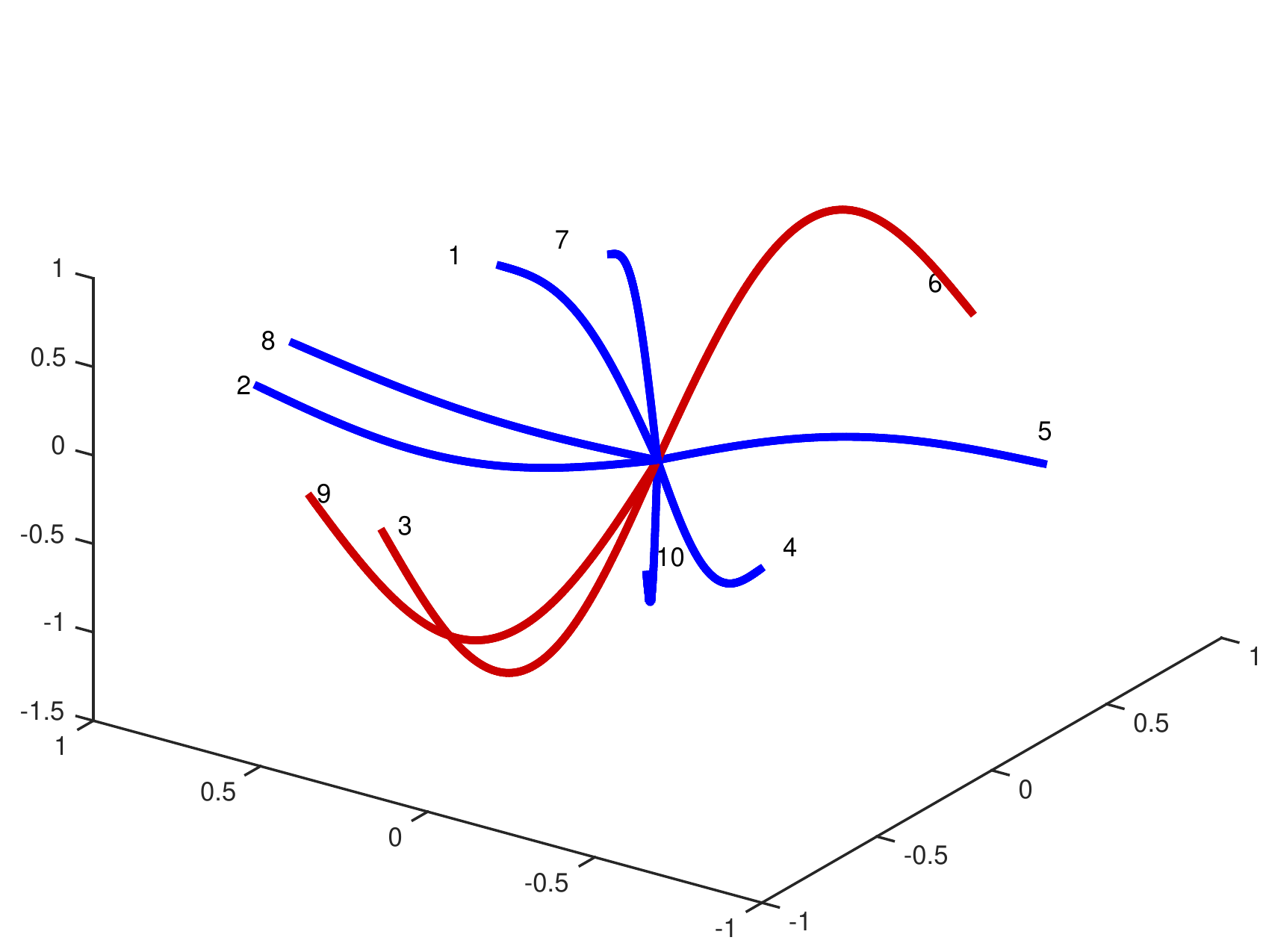} } }
        \end{subfigure} 
\qquad
        \begin{subfigure}[Solution $p^{100}$, $n = 100$.]
                {\resizebox{7cm}{7cm}
{\includegraphics{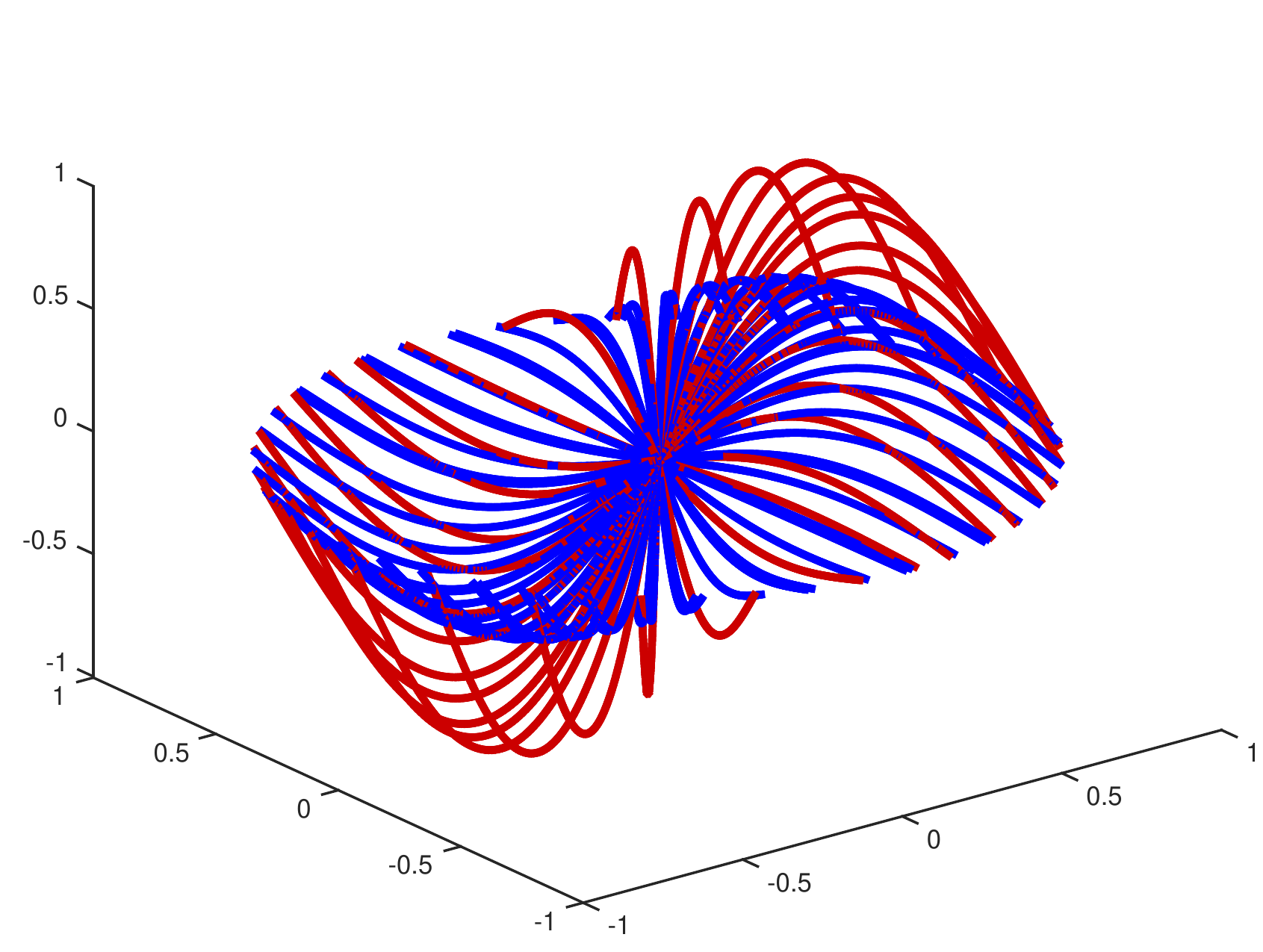} } }
        \end{subfigure} 
%
%
\caption{\textbf{Solutions \textit{Example} \ref{Ex. Basic Example}}. Diffusion coefficient $K_{d}$, see \eqref{Def Deterministic Experimental Difussion Coefficient}. The solutions depicted in figures (a) and (b) on the edges $v_{\ell}v_{0}$ are colored with red if $K_{d}(v_{\ell}v_{0}) = 1$ (i.e., $\ell \equiv 0 \mod 3$), or blue if $K_{d}(v_{\ell}v_{0}) = 2$ (i.e., $\ell \not\equiv 0 \mod 3$). Forcing term $F:\Omega\rightarrow \R$, $F(t \cos \ell, t\sin \ell)\defining \pi^{2}\,  \sin (\pi t) \cos(\ell)$. \label{Fig Graphs First Example}} 
\end{figure}
\end{example}
\begin{example}[Probabilistic Flexibilities for \textit{Example} \ref{Ex. Basic Example}]\label{Ex. Probabilistic Flexibilities}
This experiment follows the observations of \textit{Remark} \ref{Rem Probabilistic Flexibilities}. In this case take $X\defining K_{p}$, defined in \eqref{Def Probabilistic Experimental Difussion Coefficient}. Let $Z: \N\rightarrow [-100, 100]$ be a random variable with uniform distribution and define $Y: \Omega_{G}\rightarrow \R$ by $Y(t \cos \ell, t\sin \ell)\defining \pi^{2}\,  \sin (\pi t) \cos(\ell) + Z(\ell)$. It is direct to see that $X$ and $Y$ satisfy \textit{Hypothesis} \ref{Hyp Forcing Term Cesaro Means Assumptions} and due to the Law of Large Numbers, they also satisfy \eqref{Eq Probabilistic Diffusion Coefficient of Finite Range} and \eqref{Eq Random Forcing Vector} respectively. Therefore
\begin{equation*}
\thickbar{Y}_{1} = 
\mrad[F\vert_{\Omega_{G}^{1}}]=
\thickbar{Y}_{2} = 
\mrad[F\vert_{\Omega_{G}^{2}}] = 
\mrad[F]  = \p_{1} = \p_{2} = 0.
\end{equation*}
The following table is the summary for a fixed realization of $X$ (to keep the edge coloring consistent) and different realizations of $Y$ on each stage. Convergence is observed, as expected it is slower than in the previous case. This would also occur for different realizations of $X$ and $Y$ simultaneously.  
%
%
\begin{table*}[h!]
\centerline{\textit{Example} \ref{Ex. Probabilistic Flexibilities}: Convergence Table, $K = K_{p}$.}\vspace{5pt}
\def\arraystretch{1.4}
\begin{center}
\begin{tabular}{ c c c c c }
    \hline
$n$  & $\Vert  \phone- \p_{1}  \Vert_{L^{2}(e_{1})} $ & $\Vert  \phtwo- \p_{2}  \Vert_{L^{2}(e_{1})}$ & $\Vert  \phone- \p_{1}  \Vert_{H^{1}_{0}(e_{2})}$ & $\Vert  \phtwo- \p_{2}  \Vert_{H^{1}_{0}(e_{2})}$ \\ 
    \toprule
10 &  0.5534 & 0.0938 & 1.6629 & 0.5381  \\
20 & 0.0965  & 0.1594 & 0.5186 &  0.3761    \\
100 & 0.0653 & 0.1322 & 0.3809 &  0.2569    \\
1000  & 0.0201  & 0.0302 & 0.0658 &  0.0597   \\ 
    \hline
\end{tabular}
\end{center}
\end{table*}
\begin{figure}[h] 
\vspace{-1.2cm}
        \centering
        \begin{subfigure}[Solution $p^{20}$, $n = 20$.]
                {\resizebox{7cm}{7cm}
                {\includegraphics{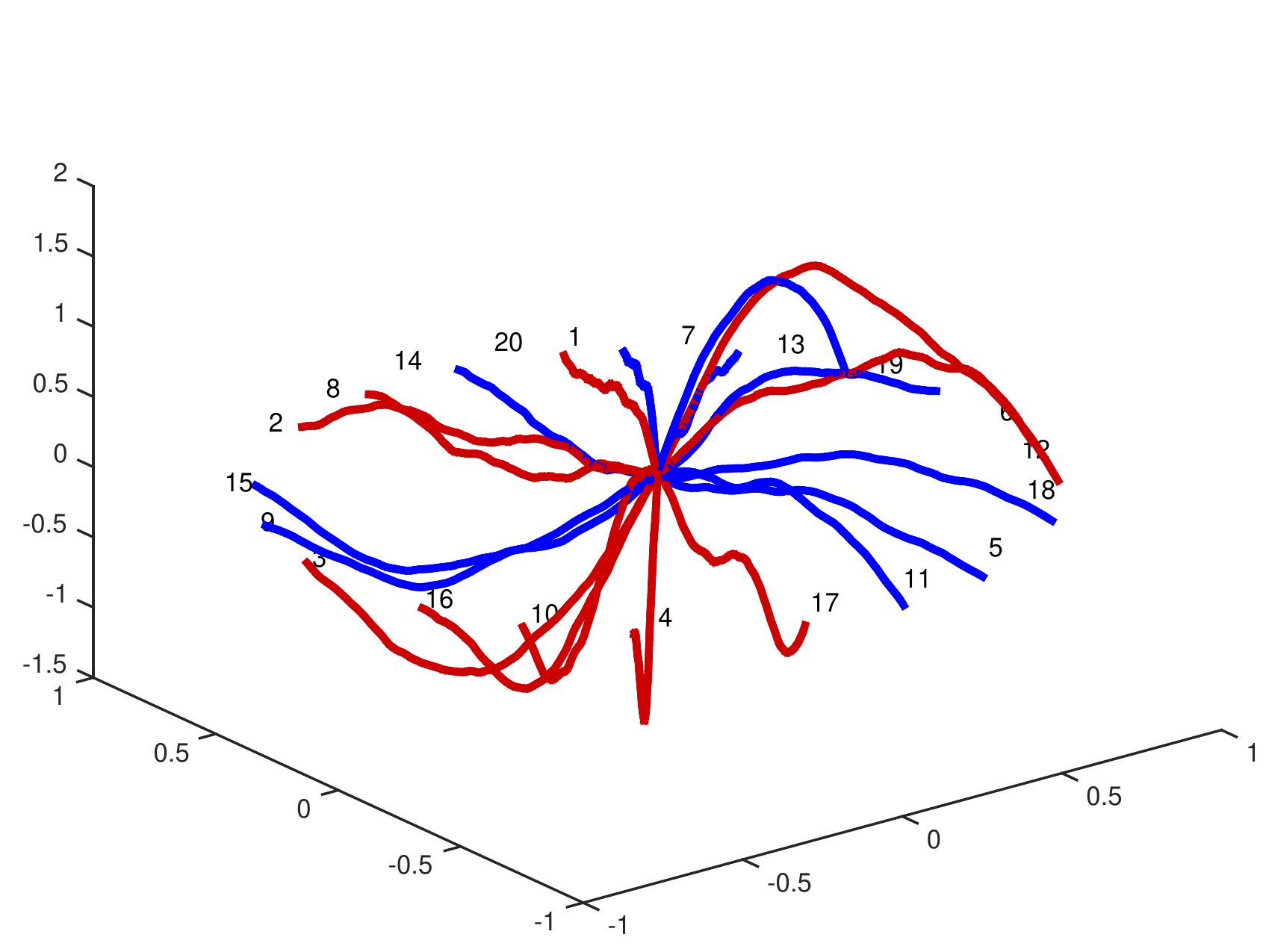} } }
        \end{subfigure} 
        \qquad
        ~ 
          \begin{subfigure}[Solution $p^{50}$, $n = 50$.]
                {\resizebox{7cm}{7cm}
                {\includegraphics{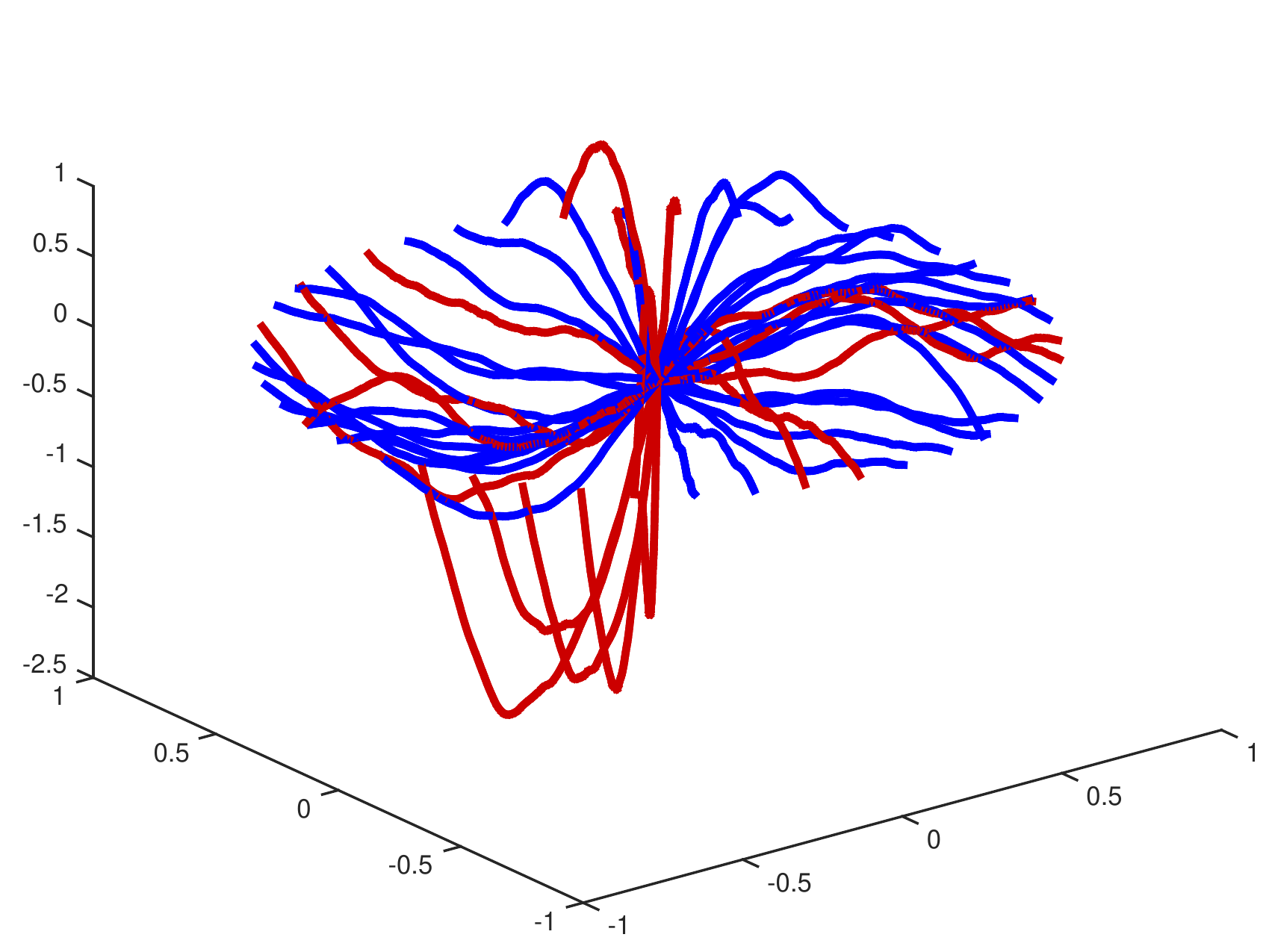} } }                
        \end{subfigure} 
%
%
%
\caption{\textbf{Solutions of \textit{Example} \ref{Ex. Probabilistic Flexibilities}} Fixed realization of diffusion coefficient $K_{p}$, see \eqref{Def Probabilistic Experimental Difussion Coefficient}.
Forcing term $Y:\Omega_{G}\rightarrow \R$, $Y(t \cos \ell, t\sin \ell)\defining \pi^{2}\,  \sin (\pi t) \cos(\ell) + Z(\ell)$, with $Z: \N\rightarrow [-100, 100]$, random variable and $Z\sim \text{uniformly}$. Different realizations for $Y$ on each stage. The solutions depicted in figures (a) and (b) on the edges $v_{\ell}v_{0}$ are colored with red if $K_{p}(v_{\ell}v_{0}) = 1$ ($\Exp[K_{p} = 1] = \frac{1}{3}$), or blue colored if $K_{p}(v_{\ell}v_{0}) = 2$ ($\Exp[K_{p} = 2] = \frac{2}{3}$). \label{Fig Graphs Second Example} }
\end{figure}
\end{example}
\begin{example}[A non-Riemann Integrable Forcing Term]\label{Ex. non-Riemann Integrable Forcing Terms}
For our final theoretical example we use a non-Riemann Integrable forcing term. Moreover, the following function is highly oscillatory inside each subdomain $\Omega_{G}^{1}$ and ${\Omega_{G}^{2}}$, and it can not be seen as Riemann integrable when restricted to any of the sub-domains. Let $F: \Omega_{G}\rightarrow \R$ be defined by
\begin{equation}\label{Eq non-Riemann Integrable Forcing Terms}
F(t \cos \ell, t\sin \ell)\defining \begin{cases}
4\pi^{2}\,  \sin (2\pi t) + (-1)^{\lfloor \frac{\ell}{6} \rfloor}\times 10 \times (\ell - \big\lfloor \frac{\ell}{2\pi} \big\rfloor) , & \ell \equiv 0 \mod 3, \\
\pi^{2}\,  \sin (\pi t) + (-1)^{\lfloor \frac{\ell}{6} \rfloor}\times 10 \times (\ell - \big\lfloor \frac{\ell}{2\pi} \big\rfloor)  , & \ell \not\equiv 0 \mod 3 .
\end{cases}
\end{equation}
On one hand, both sequences $\{v_{\ell}: \ell \in \N, \ell \equiv  0 \mod 3 \}$ and $\{v_{\ell}: \ell \in \N, \ell \not\equiv 0 \mod 3 \}$ are equidistributed. On the other hand both parts of the forcing term, the radial and the angular are Ces\`aro convergent on each $\Omega_{G}^{i}$ for $i = 1, 2$. The Ces\`aro average of the angular summand is zero on $\Omega_{G}^{i}$ for $i = 1, 2$. In contrast, the radial summand can be seen as Riemann integrable separately on each $\Omega_{G}^{i}$ for $i = 1, 2$, therefore, due to \textit{Weyl's Theorem} \ref{Th Weyl's Theorem} its Ces\`aro average is given by $\F_{1} = \mrad[F\vert_{\Omega_{G}^{1}}]$ and $\F_{2} = \mrad[F\vert_{\Omega_{G}^{2}}]$; more explicitly, 
%
%
\begin{align}\label{Eq Averaged non-Riemann Integrable Forcing Terms}
& \F_{1}(t) =
(2\pi)^{2}\sin (2\pi t) \, ,&
& \F_{2}(t) = 
\pi^{2}\, \sin (\pi t)  .
\end{align}
For this case the exact solution $\p = \p \, \ind_{\sigma_{1}} + \p \, \ind_{\sigma_{2}} \in H_{0}^{1}(-1, 1)$ of the upscaled \textit{Problem} \eqref{Stmt limit problem} is given by
\begin{align}\label{Eq Averaged Solutions non-Riemann Integrable Forcing Term}
& \p_{1}(t) =
\sin (2\pi t) \, ,&
& \p_{2}(t) = 
\frac{1}{2} \, \sin (\pi t)  .
\end{align}
We summarize the convergence behavior in the table below.
%
%
\begin{table*}[h!]
\centerline{\textit{Example} \ref{Ex. non-Riemann Integrable Forcing Terms}: Convergence Table, $K = K_{d}$.}\vspace{5pt}
\def\arraystretch{1.4}
\begin{center}
\begin{tabular}{ c c c c c }
    \hline
$n$  & $\Vert  \phone- \p_{1}  \Vert_{L^{2}(e_{1})} $ & $\Vert  \phtwo- \p_{2}  \Vert_{L^{2}(e_{1})}$ & $\Vert  \phone- \p_{1}  \Vert_{H^{1}_{0}(e_{2})}$ & $\Vert  \phtwo- \p_{2}  \Vert_{H^{1}_{0}(e_{2})}$ \\ 
    \toprule
10 &  1.6392 & 0.4900 & 5.7447 & 1.3210  \\
20 & 0.4127  & 0.9305 & 1.8930 &  1.7782    \\
100 & 0.2125 & 0.3312 & 0.4986 &  0.6275    \\
1000  & 0.0138  & 0.0189 & 0.0852 &  0.0371   \\ 
    \hline
\end{tabular}
\end{center}
\end{table*}
\begin{figure}[h] 
\vspace{-1.3cm}
        \centering
        ~ 
          \begin{subfigure}[Solution $p^{20}$, $n = 20$.]
                {\resizebox{7cm}{7cm}
                {\includegraphics{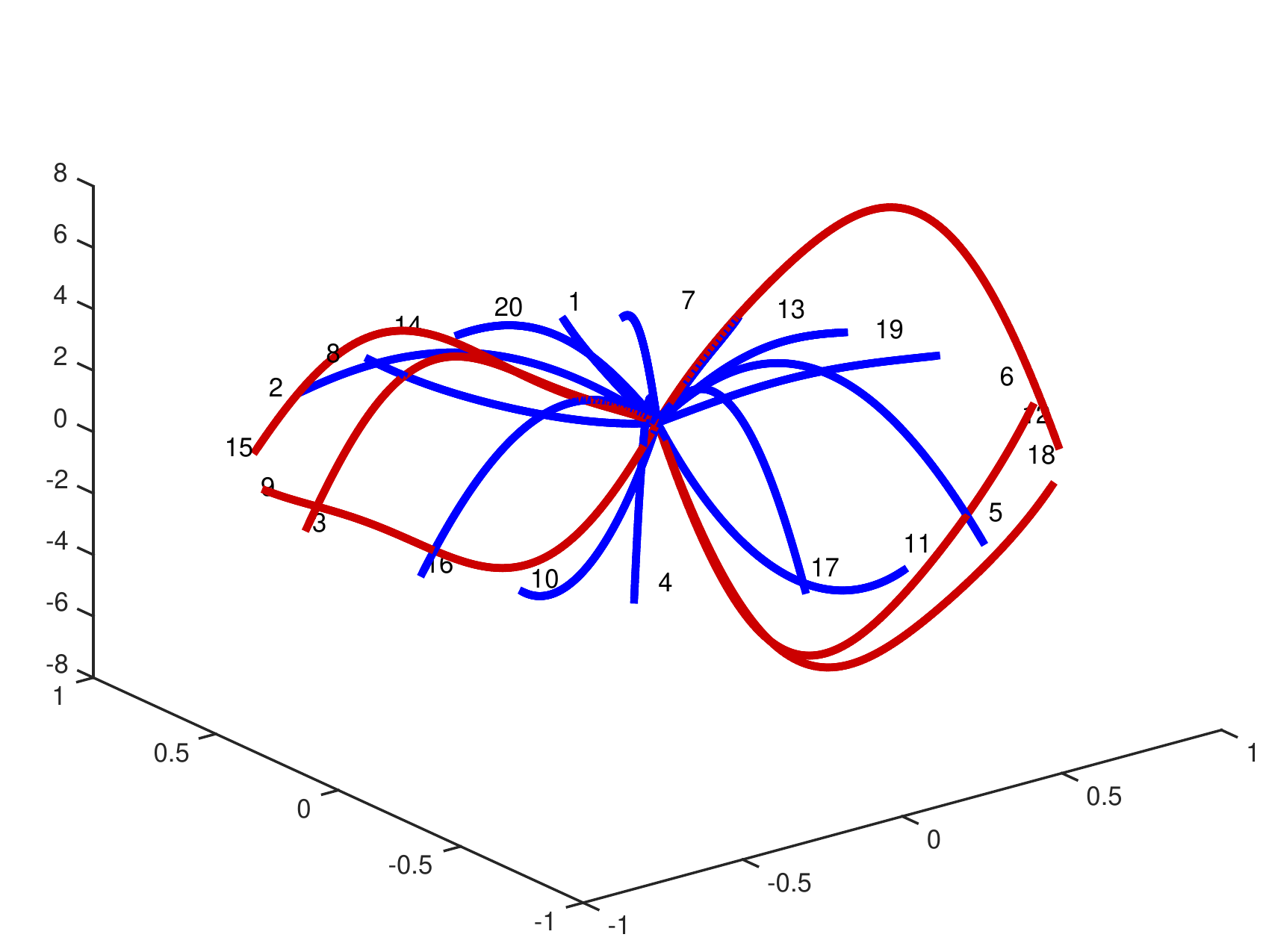} } }                
        \end{subfigure} 
\qquad
        \begin{subfigure}[Solution $p^{50}$, $n = 50$.]
                {\resizebox{7cm}{7cm}
{\includegraphics{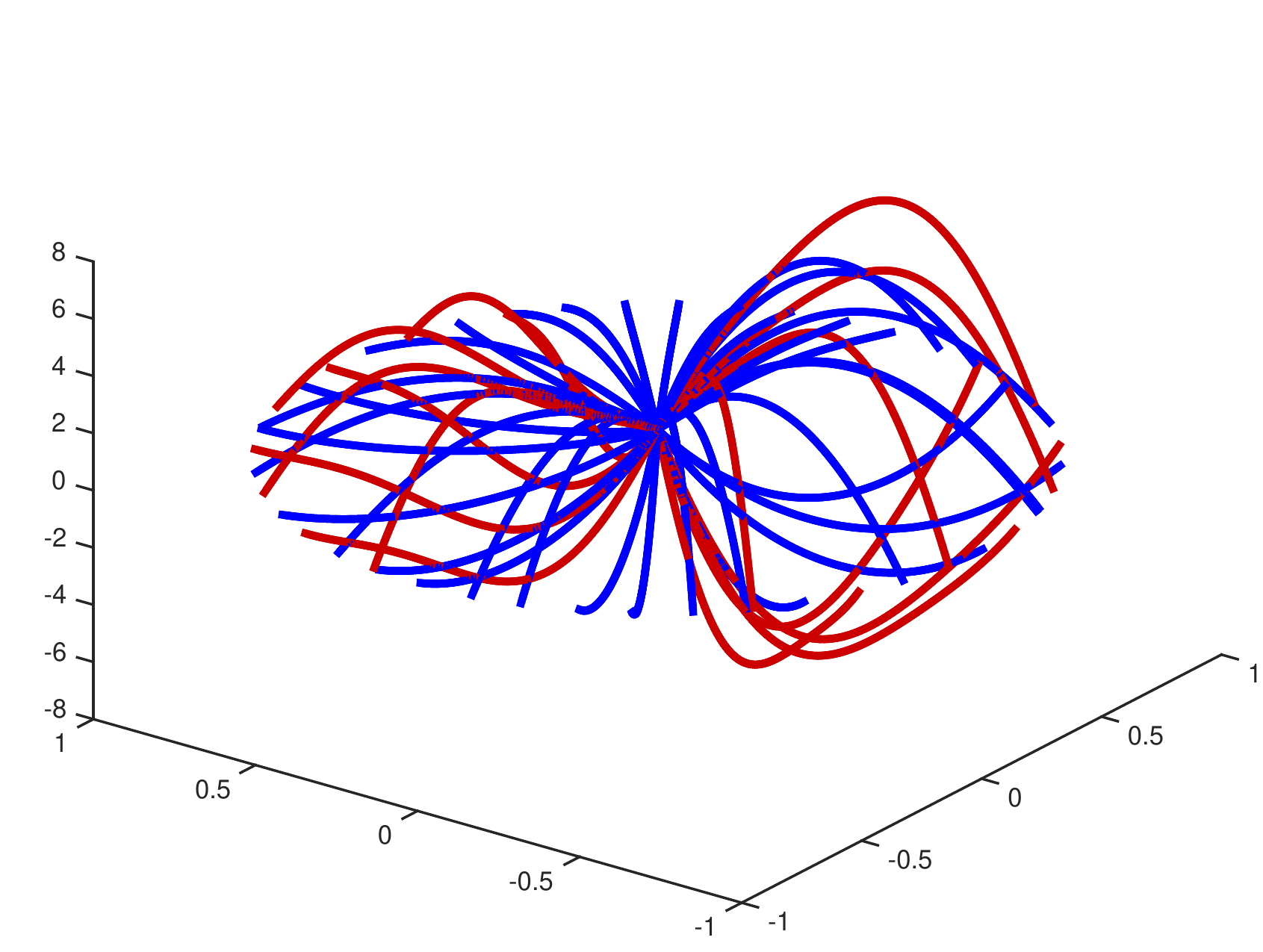} } }
        \end{subfigure} 
%
%
\caption{\textbf{Solutions of \textit{Example} \ref{Ex. non-Riemann Integrable Forcing Terms}} Diffusion coefficient $K_{d}$, see \eqref{Def Deterministic Experimental Difussion Coefficient}. The solutions depicted in figures (a) and (b) on the edges $v_{\ell}v_{0}$ are colored with red if $K_{d}(v_{\ell}v_{0}) = 1$ (i.e., $\ell \equiv 0 \mod 3$), or blue if $K_{d}(v_{\ell}v_{0}) = 2$ (i.e., $\ell \not\equiv 0 \mod 3$). Forcing term $F:\Omega_{G}\rightarrow \R$ see \eqref{Eq non-Riemann Integrable Forcing Terms}.   \label{Fig Graphs Third Example} }
\end{figure}
\end{example}
%
%
%
\subsection{Numerical Experimentation Examples}
%
%
%
%
In this section we present two examples, breaking different hypotheses of those required in the theoretical analysis discussed above. As there is not a known exact solution, we follow Cauchy's convergence criterion for the sequences $\{\p^{ \, n}_{i}: n\in \N  \}$ with $i = 1, 2$. However, we do not sample only points but intervals of observation and report the averages of the observed data. More specifically
\begin{align*}
& \epsilon^{n}_{i}\defining \frac{1}{10}\sum_{j = n - 4}^{n + 5} \Vert \p^{ \, j}_{i} - \p^{ \, j - 1}_{i} \Vert_{L^{2}(e_{i})} &   &\text{and} 
& \delta^{n}_{i}\defining\frac{1}{10}\sum_{j = n - 4}^{n + 5} \Vert \p^{ \, j}_{i} - \p^{ \, j - 1}_{i} \Vert_{H^{1}(e_{i})} , \\
& & & \text{for} 
&  i = 1, 2 \, , \;  
n = 10, 20, 100, 1000.
\end{align*}
\begin{example}[A Locally Unbounded Forcing Term]\label{Ex. Unbounded Example}
For our experiment we use a variation of \textit{Example} \ref{Ex. non-Riemann Integrable Forcing Terms}, keeping the well-behaved radial part but adding an unbounded angular part, which is known to be Ces\`aro convergent to zero. Consider the forcing term $F: \Omega_{G}\rightarrow \R$ defined by
\begin{equation}\label{Eq Unbounded Forcing Terms}
F(t \cos \ell, t\sin \ell)\defining \begin{cases}
4\pi^{2}\,  \sin (2\pi t) + (-1)^{\ell} \sqrt{\ell} \,  , & \ell \equiv 0 \mod 3, \\
\pi^{2}\,  \sin (\pi t) + (-1)^{\ell} \sqrt{\ell} \,  , & \ell \not\equiv 0 \mod 3 .
\end{cases}
\end{equation}
Clearly, $\sup_{e\in E} \Vert F \Vert_{L^{2}(e)} = \infty$ i.e.,  \textit{Hypothesis} \ref{Hyp Forcing Terms and Permeability}-(i) is not satisfied. 
It is not hard to adjust the techniques presented in \textit{Section} \ref{Sec Estimates and Cesaro Convergence} to this case, when the forcing term is Ces\`aro convergent without satisfying the condition $\sup_{e\in E} \Vert F \Vert_{L^{2}(e)} = \infty$; however the properties of edgewise uniform convergence of \textit{Section} \ref{Estimates and Edgewise Convergence Statements} can not be concluded. Consequently, we observe the following convergence behavior.
%
%
\begin{table*}[h!]
\centerline{\textit{Example} \ref{Ex. Unbounded Example}: Convergence Table, $K = K_{d}$.}\vspace{5pt}
\def\arraystretch{1.2}
\setlength\tabcolsep{0.8cm}
\begin{center}
\begin{tabular}{ c c c c c }
    \hline
$n$  & $\epsilon^{n}_{1}$  & 
$\epsilon^{n}_{2}$ & 
$\delta^{n}_{1}$ & 
$\delta^{n}_{2}$ \\ 
    \toprule
10 &  0.1547 & 0.1578 & 2.7336 & 2.7338  \\
20 & 0.0618  & 0.0645 & 1.0734 &  1.0747    \\
100 & 0.0277 & 0.0224 & 0.3394 &  0.3320    \\
1000  & 0.0086  & 0.0065 & 0.0984 &  0.0955   \\ 
    \hline
\end{tabular}
\end{center}
\end{table*}
%
%
%
%
%
%

%
%
\begin{figure}[h] 
\vspace{-1.3cm}
        \centering
        \begin{subfigure}[Solution $p^{20}$, $n = 20$.]
                {\resizebox{7cm}{7cm}
                {\includegraphics{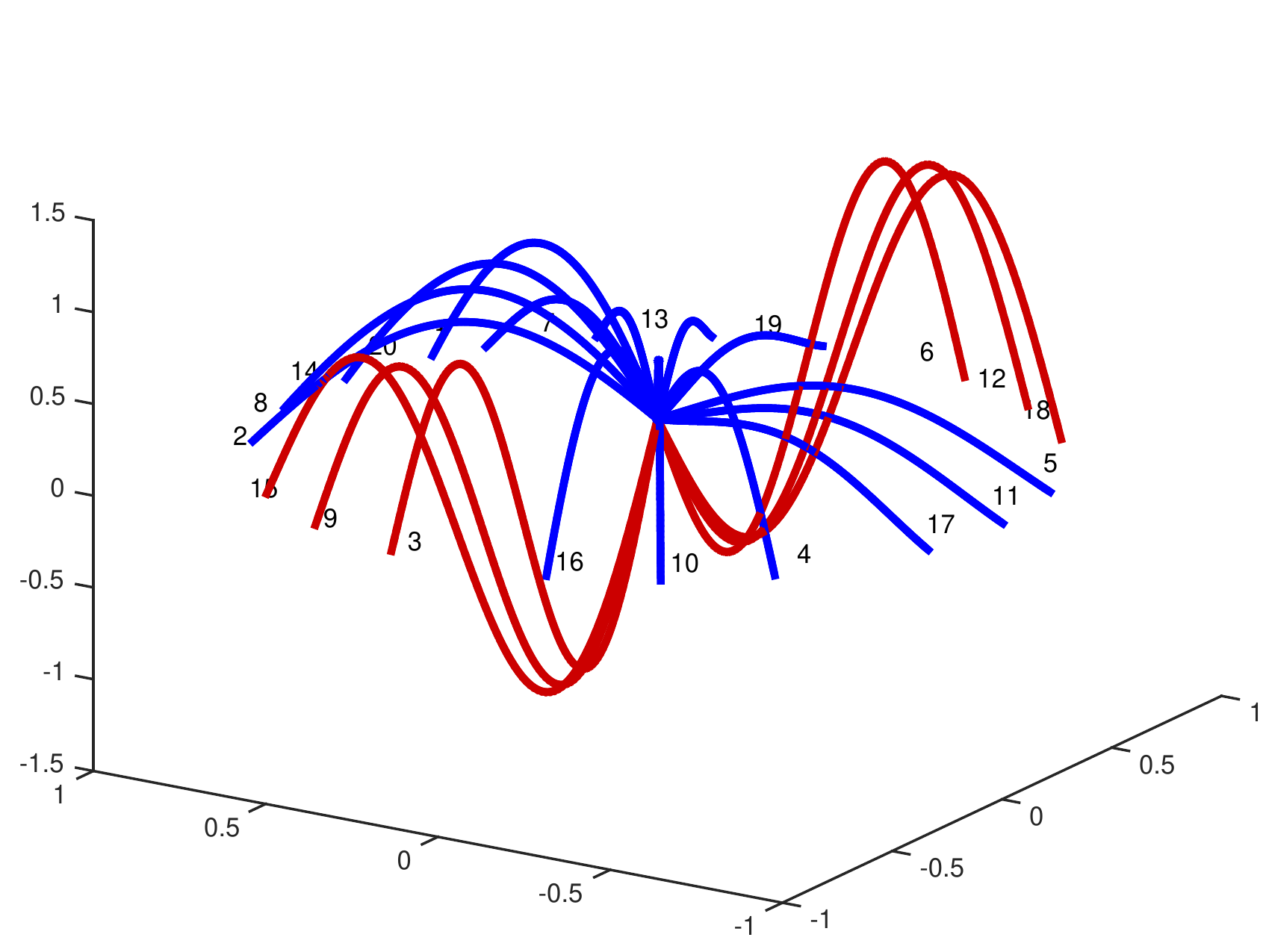} } }
        \end{subfigure} 
        \qquad
        ~ 
          \begin{subfigure}[Solution $p^{100}$, $n = 100$.]
                {\resizebox{7cm}{7cm}
                {\includegraphics{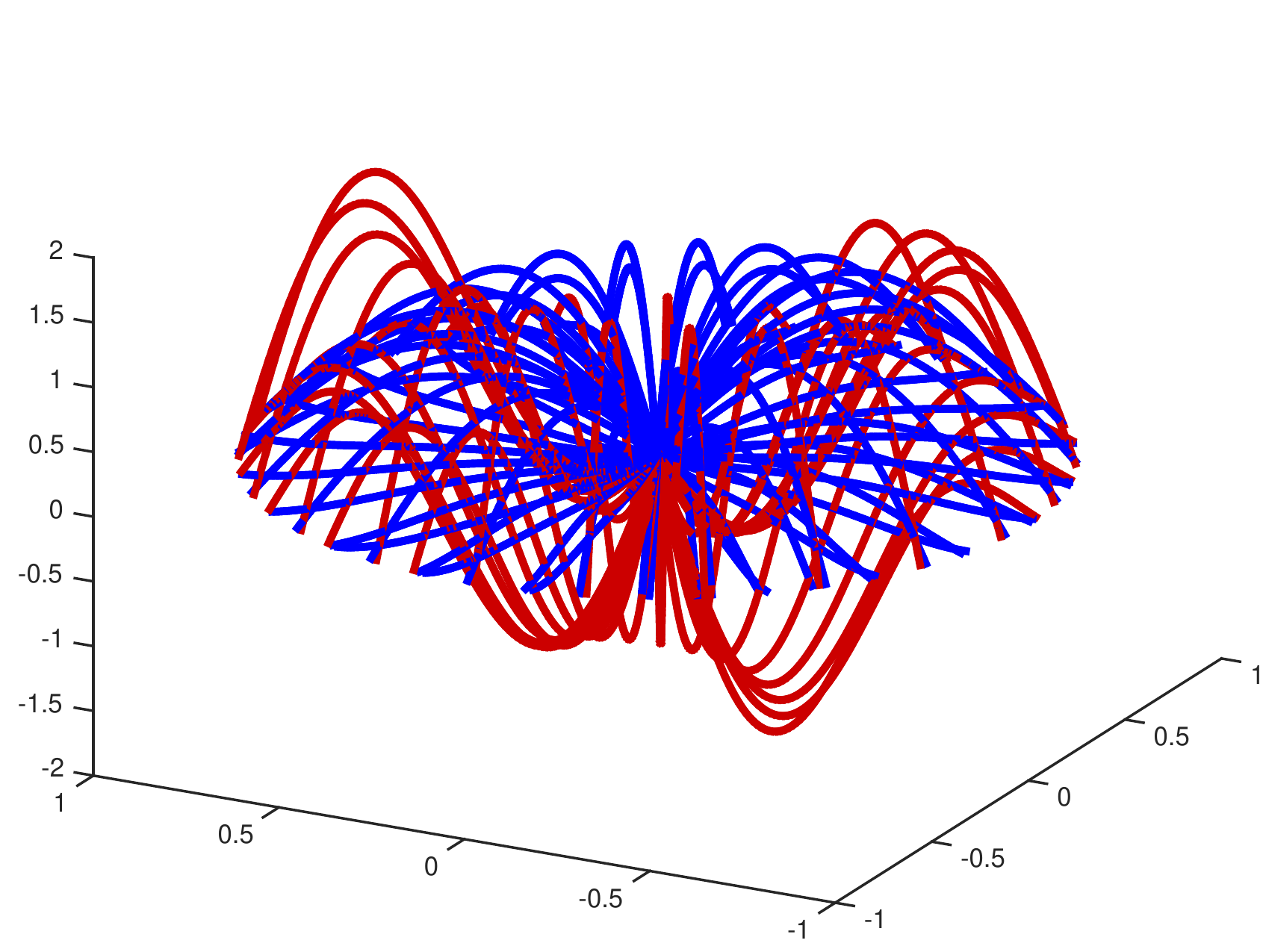} } }                
        \end{subfigure} 
%
%
%
\caption{\textbf{Solutions of \textit{Example} \ref{Ex. Unbounded Example}} Diffusion coefficient $K_{d}$, see \eqref{Def Deterministic Experimental Difussion Coefficient}. The solutions depicted in figures (a) and (b) on the edges $v_{\ell}v_{0}$ are colored with red if $K_{d}(v_{\ell}v_{0}) = 1$ (i.e., $\ell \equiv 0 \mod 3$), or blue if $K_{d}(v_{\ell}v_{0}) = 2$ (i.e., $\ell \not\equiv 0 \mod 3$). Forcing term $F:\Omega_{G}\rightarrow \R$ see \eqref{Eq Unbounded Forcing Terms}.  \label{Fig Graphs Fourth Example} }
\end{figure}
\end{example}
\begin{example}[A Forcing Term with Unbounded Frequency Modes]\label{Ex. Unbounded Frequency}
For our last experiment we use a variation of \textit{Example} \ref{Ex. non-Riemann Integrable Forcing Terms}, keeping it bounded, but introducing unbounded frequencies. Consider the forcing term $F: \Omega_{G}\rightarrow \R$ defined by
\begin{equation}\label{Eq Unbounded Frequency Forcing Terms}
F(t \cos \ell, t\sin \ell)\defining \begin{cases}
4\pi^{2}\,  \sin (2\pi t\cdot \ell) \,  , & \ell \equiv 0 \mod 3, \\
\pi^{2}\,  \sin (\pi t\cdot \ell) \,  , & \ell \not\equiv 0 \mod 3 .
\end{cases}
\end{equation}
Clearly, $F$ verifies \textit{Hypothesis} \ref{Hyp Forcing Terms and Permeability} , then \textit{Lemma} \ref{Th Estimates on v_0 = 0} implies edgewise uniform convergence of the solutions,
however \textit{Hypothesis}-(ii) \ref{Hyp Forcing Term Cesaro Means Assumptions} is not satisfied. Therefore, we observe that the whole sequence is not Cauchy, although it has Cauchy subsequences as the following table shows.
%
%
\begin{table*}[h!]
\centerline{\textit{Example} \ref{Ex. Unbounded Frequency}: Convergence Table, $K = K_{d}$.}\vspace{5pt}
\def\arraystretch{1.2}
\setlength\tabcolsep{0.8cm}
\begin{center}
\begin{tabular}{ c c c c c }
    \hline
$n$  & $\epsilon^{n}_{1}$  & 
$\epsilon^{n}_{2}$ & 
$\delta^{n}_{1}$ & 
$\delta^{n}_{2}$ \\ 
    \toprule
10 &  0.0264 & 0.0267 & 0.4157 & 0.3835  \\
20 & 0.0078  & 0.0089 & 0.1342 &  0.1327    \\
100 & 0.0004 & 0.0005 & 0.0077 &  0.0076    \\
500 & 0.00004 & 0.00004 & 0.00073 &  0.00072    \\
1000  & 0.00066  & 0.00049 & 0.0081 &  0.0078   \\ 
1200  & 0.00004  & 0.00005 & 0.000787 &  0.000786   \\ 
    \hline
\end{tabular}
\end{center}
\end{table*}
\begin{figure}[h] 
\vspace{-0.9cm}
        \centering
        ~ 
          \begin{subfigure}[Solution $p^{500}$, $n = 500$.]
                {\resizebox{7cm}{7cm}
                {\includegraphics{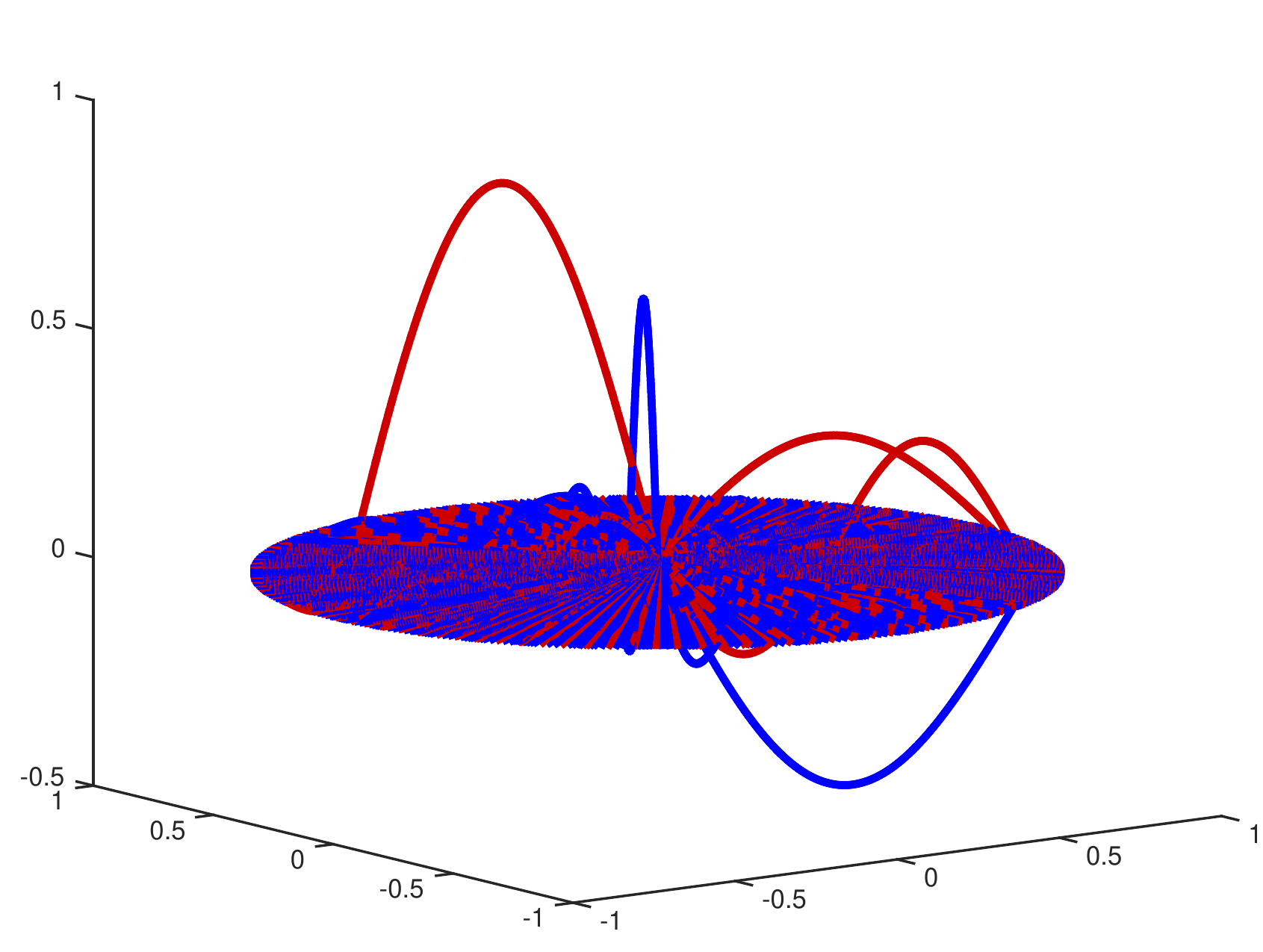} } }                
        \end{subfigure} 
\qquad
        \begin{subfigure}[Solution $p^{1000}$, $n = 1000$.]
                {\resizebox{7cm}{7cm}
{\includegraphics{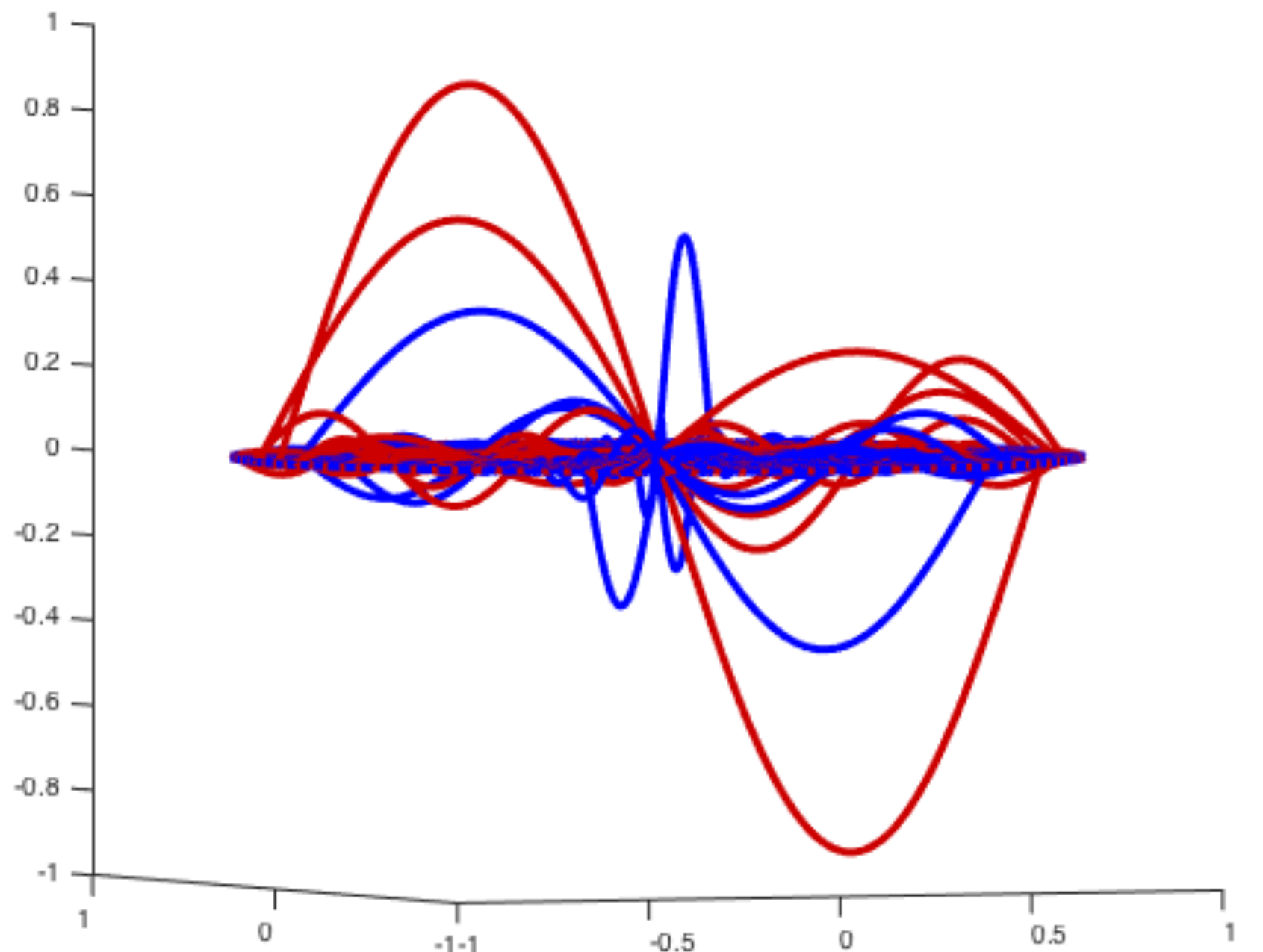} } }
        \end{subfigure} 
%
%
\caption{\textbf{Solutions of \textit{Example} \ref{Ex. Unbounded Frequency}}. The  Diffusion coefficient $K_{d}$, see \eqref{Def Deterministic Experimental Difussion Coefficient}. The solutions depicted in figures (a) and (b) on the edges $v_{\ell}v_{0}$ are colored with red if $K_{d}(v_{\ell}v_{0}) = 1$ (i.e., $\ell \equiv 0 \mod 3$), or blue if $K_{d}(v_{\ell}v_{0}) = 2$ (i.e., $\ell \not\equiv 0 \mod 3$). Forcing term $F:\Omega_{G}\rightarrow \R$ see \eqref{Eq Unbounded Frequency Forcing Terms}.  \label{Fig Graphs Fifth Example} }
\end{figure}

It follows that this system has more than one internal equilibrium. Consequently, an upscaled model of a system such as this, should contain uncertainty which, in this specific case, remains bounded due to the properties of the forcing term $F$.  
\end{example}
%
%
%
\subsection{Closing Observations}\label{Sec Numerical Closing Observations}
%
%
\begin{enumerate}[(i)]
\item 
The authors tried to find experimentally a rate of convergence using the well-know estimate 
\begin{align*}
& \alpha_{i} \sim \frac{\log \Vert  \p^{ \, n+1}_{i}- \p^{ \, n }_{i}  \Vert - 
\log   \Vert  \p^{ \, n}_{i} - \p^{ \, n-1}_{i} \Vert }{
\log   \Vert  \p^{ \, n}_{i} - \p^{ \, n-1}_{i} \Vert
- \log \Vert  \p^{ \, n-1}_{i}- \p^{ \, n -2}_{i}  \Vert } \, ,&
& i = 1, 2.
\end{align*}
The sampling was made on the intervals $n-5\leq j \leq n + 5$, for $n = 10, 20, 100, 500$ and $1000$. Experiments were run on all the examples except for \textit{Example} \ref{Ex. Unbounded Frequency}. In none of the cases, solid numerical evidence was detected that could suggest an order of convergence for the phenomenon. 

\item Experiments for random variations of the examples above were also executed, under the hypothesis that random variables were subject to the Law of Large Numbers. Convergence, slower than its corresponding deterministic version was observed, as expected. This is important for its applicability to upscaling networks derived from game theory, see \cite{Jackson}.  

\end{enumerate}
%
%
%
%
%
%
%
\section{Conclusions and Final Discussion}\label{Conclusions and Final Discussion}
%
%
The present work yields several accomplishments and also limitations as we point out below.
\begin{enumerate}[(i)]
\item The method presented in this paper can be easily extended to general scale-free networks in a very simple way. First identify the communication kernel (see \cite{Kim}). Second, for each node in the kernel, replace its numerous incident low-degree nodes by the upscaled nodes together with the homogenized diffusion coefficients and forcing terms, see \textit{Figure} \ref{Fig Networks}.

\item  The particular scale-free network treated in the paper i.e., the star metric graph, arises naturally in some important examples. These come from the theory of the strategic network formation, where the agents choose their connections following utilitarian behavior. Under certain conditions for the benefit-cost relation affecting the actors when establishing links with other agents, the asymptotic network is star-shaped (see \cite{Jacksonbloch}).

\item The scale-free networks are frequent in many real world examples as already mentioned. It follows that the method is applicable to a wide range of cases. However, important networks can not be treated the same way for homogenization, even if they share some important properties of communication. The small-world networks constitute an example since they are highly clustered, this feature contradicts the power-law degree distribution hypothesis. See \cite{Jacksonrogers} for a detailed exposition on the matter.

\item The upscaling of the diffusion phenomenon is done in a hybrid fashion. On one hand, the diffusion on the low-degree nodes is modeled by the weak variational form of the differential operators defined over the graph, but ignoring its combinatorial structure. On the other hand, the diffusion on the communication kernel will still depend on both, the differential operators and the combinatorial structure. This is an important achievement, because it is consistent with the nature of available data for the analysis of real world networks. Typically, the data for central (or highly connected) agents are more reliable than data for marginal (or low degree) agents. 

\item The central Ces\`aro convergence hypotheses for data behavior (stated  in \textit{Lemma} \ref{Th Estimates on v_0 = 0}-(iii), as well as those contained in \textit{Hypothesis} \ref{Hyp Forcing Term Cesaro Means Assumptions}, in order to conclude convergence have probabilistic-statistical nature. This is one of the main accomplishments of the work, because the hypotheses are mild and adjust to realistic scenarios; unlike strong hypotheses of topological nature such as periodicity, continuity, differentiability or even Riemann-integrability of the forcing terms (see \cite{Hornung}). This fact is further illustrated in \textit{Example} \ref{Ex. non-Riemann Integrable Forcing Terms}, where good asymptotic behavior is observed for a forcing term which is nowhere continuous on the domain $\Omega_{G}$ of analysis.

\item An important and desirable consequence of the data hypotheses adopted, is that the method can be extended to more general scenarios, as mentioned in \textit{Remark} \ref{Rem Probabilistic Flexibilities}, reported in \textit{Subsection} \ref{Sec Numerical Closing Observations} and illustrated in \textit{Examples} \ref{Ex. Probabilistic Flexibilities}, \ref{Ex. Unbounded Example} and \ref{Ex. Unbounded Frequency}. Moreover, \textit{Example} \ref{Ex. Unbounded Frequency} suggests a probabilistic upscaled model for the communication kernel, to be explored in future work.   

\item A different line of future research consists in the analysis of the same phenomenon, but using the mixed-mixed variational formulation introduced in \cite{MoralesShow2} instead of the direct one used in the present analysis. The key motivation in doing so, is that the mixed-mixed formulation is capable of modeling more general exchange conditions than those handled by the direct variational formulation and by the classic mixed formulations. This advantage can broaden in a significant way the spectrum of real-world networks which can be successfully modeled and upscaled.

\item Finally, the preexistent literature typically analyses the asymptotic behavior of diffusion in complex networks, starting from fully discrete models (e.g., \cite{JacksonYariv, JacksonLopez}). The pseudo-discrete treatment that we have followed here, constitutes more a complementary than an alternative approach. Depending on the availability of data and/or sampling, as well as the scale of interest for a particular problem, it is natural to consider a ``blending" of both techniques.

\end{enumerate}
\section*{Acknowledgements}
The authors wish to acknowledge Universidad Nacional de Colombia, Sede Medell\'in for its support in this work through the project HERMES 27798. The authors also wish to thank Professor Ma\l{}gorzata Peszy\'nska from Oregon State University, for authorizing the use of code \textbf{fem1d.m} \cite{PeszynskaFEM} in the implementation of the numerical experiments presented in \textit{Section} \ref{Sec Numerical Experiments}. It is a tool of remarkable quality, efficiency and versatility that has proved to be a decisive element in production and shaping of this work. 
%
%
%
%

%
%
\end{document}